\documentclass[leqno]{amsart}
\usepackage{amsmath,amsthm,amssymb,latexsym,graphicx, mathtools, enumerate,bm,cite,mathrsfs}
\usepackage{color} % temporary

\newcommand{\NN}{{\mathbb N}}

\newcommand{\RR}{{\mathbb R}}

\newcommand{\abs}[1]{ \left| #1 \right|}

\newcommand{\del}{\partial}

\newcommand{\eps}{\varepsilon}

\newcommand{\nor}[2]{\left\|#1\right\|_{#2}}
\newcommand{\oline}[1]{\overline{#1}}
\newcommand{\oo}{\infty}
\newcommand{\pars}[1]{\left(#1\right)}

\newcommand{\tr}{\on{tr}}

\DeclareMathOperator*{\osc}{osc}

\newcommand{\mcl}{\mathcal}
\newcommand{\mbb}{\mathbb}
\newcommand{\mbf}{\mathbf}
\newcommand{\on}{\operatorname}

\newtheorem{lemma}{Lemma}

\newtheorem{theorem}{Theorem}

\newtheorem{definition}{Definition}

\numberwithin{equation}{section}
\numberwithin{lemma}{section}
\numberwithin{proposition}{section}
\numberwithin{theorem}{section}
\numberwithin{corollary}{section}
\numberwithin{definition}{section}

\begin{document}
\title{H\"older regularity of Hamilton-Jacobi equations with stochastic forcing}
\author{Pierre Cardaliaguet$^{1,2}$ and Benjamin Seeger$^{1,3,4}$}
\address{$^1$Universit\'e Paris-Dauphine \& Coll\`ege de France \\ Place du Mar\'echal de Lattre de Tassigny \\ 75016 Paris, France}
\email{$^2$cardaliaguet@ceremade.dauphine.fr, $^3$seeger@ceremade.dauphine.fr}

\thanks{$^4$Partially supported by the National Science Foundation Mathematical Sciences Postdoctoral Research Fellowship under Grant Number DMS-1902658
}


\date{\today}
\date{\today}

\maketitle

\begin{abstract}
	We obtain space-time H\"older regularity estimates for solutions of first- and second-order Hamilton-Jacobi equations perturbed with an additive stochastic forcing term. The bounds depend only on the growth of the Hamiltonian in the gradient and on the regularity of the stochastic coefficients, in a way that is invariant with respect to a hyperbolic scaling.
\end{abstract}

\section{Introduction}

The objective of this paper is to study the H\"older regularity of stochastically perturbed equations of the form
\begin{equation}\label{E:main1}
	du + H(Du,x,t)dt = f(x) \cdot dB
\end{equation}
and
\begin{equation}\label{E:main2}
	du + F(D^2 u, Du, x,t)dt = f(x) \cdot dB,
\end{equation}
where $H: \RR^d \times \RR^d \times \RR \to \RR$ and $F: \mbb S^d \times \RR^d \times \RR^d \times \RR\to \RR$ are coercive in $Du$, $F$ is degenerate elliptic in $D^2 u \in \mbb S^d$, $\mbb S^d$ is the space of symmetric $d \times d$ matrices, $f \in C_b^2(\RR^d,\RR^m)$, and $B$ is an $m$-dimensional Brownian motion defined over a fixed probability space $(\Omega, \mbf F, \mbf P)$.

More precisely, we are interested in the regularizing effect that comes about from the coercivity in the $Du$-variable. The goal is to show that bounded solutions of \eqref{E:main1} and \eqref{E:main2} are locally H\"older continuous with high probability, with a H\"older bound and exponent that are independent of the regularity of $H$ or $F$ in $(x,t)$, or the ellipticity in the $D^2 u$-variable. 

A major motivation for this paper is to study the average long-time, long-range behavior of solutions of \eqref{E:main1} and \eqref{E:main2} with the theory of homogenization. Specifically, if $u^\eps(x,t) := \eps u(x/\eps,t/\eps)$ for $\eps > 0$ and $(x,t) \in \RR^d \times \RR$, then $u^\eps$ solves
\begin{equation}\label{E:scaledmain1}
	du^\eps + H\pars{ Du^\eps, \frac{x}{\eps}, \frac{t}{\eps}}dt = \eps^{1/2} f\pars{\frac{x}{\eps}} \cdot dB^\eps
\end{equation}
or
\begin{equation}\label{E:scaledmain2}
	du^\eps + F\pars{  \eps D^2 u^\eps , Du^\eps, \frac{x}{\eps},\frac{t}{\eps} }dt = \eps^{1/2} f\pars{ \frac{x}{\eps}} \cdot dB^\eps,
\end{equation}
where $B^\eps(t) := \eps^{1/2}B(t/\eps)$ has the same law as $B$. Observe that the new coefficients
\[
	f^\eps(x) := \eps^{1/2} f(x/\eps),
\]
which are required to be continuously differentiable in order to make sense of the equation (twice in the case of \eqref{E:scaledmain2}), blow up in $C^1(\RR^d,\RR^m)$ and $C^2(\RR^d,\RR^m)$ as $\eps \to 0$. A major contribution of this paper is to obtain estimates that, although they depend on $\nor{Df}{\oo}$ and $\nor{D^2 f}{\oo}$, are bounded independently of $\eps$, and, in fact, the probability tails of the H\"older semi-norms converge to $0$ as $\eps \to 0$.

\subsection{Main results}
We give two types of results, for both first and second order equations. The first is an interior H\"older estimate for bounded solutions on space-time cylinders. We then use this result to prove an instantaneous H\"older regularization effect for initial value problems with bounded initial data.

For $u$ defined on the cylinder
\[
	Q_1 := \oline{ B_1} \times [-1,0] := \left\{ (x,t) \in \RR^d \times \RR : |x| \le 1, \; -1 \le t \le 0 \right\},
\]
we show that $u$ is H\"older continuous on the cylinder $B_{1/2} \times [-1/2,0]$, given that $u$ is a solution of the appropriate equation, and is nonnegative and has a random upper bound, that is, for some $\mcl S : \Omega \to [0,\oo)$,
\begin{equation}\label{A:introrandombound}
	0 \le u \le \mcl S \quad \text{in } Q_1.
\end{equation}

\begin{theorem}\label{T:introfirstorder}
	Assume, for some $A > 1$, $q > 1$, and $K > 0$, that
	\begin{equation}\label{A:Hsuperlinear}
		\frac{1}{A} |p|^q - A \le H(p,x,t) \le A|p|^q + A \quad \text{for all } (p,x,t) \in \RR^d \times \RR^d \times [-1,0],
	\end{equation}
	\begin{equation}\label{A:fC1bound}
		f \in C^1(\RR^d,\RR^m), \quad \nor{f}{\oo} + \nor{f}{\oo} \cdot \nor{Df}{\oo} \le K,
	\end{equation}
	and $u$ solves \eqref{E:main1} in $Q_1$ and satisfies \eqref{A:introrandombound}. Fix $M > 0$ and $p \ge 1$. Then there exist $\alpha = \alpha(A,q) > 0$, $\sigma = \sigma(A,q) > 0$, $\lambda_0 = \lambda_0(A,K,M,q) > 0$, and $C = C(A,K,M,p,q) > 0$ such that, for all $\lambda \ge \lambda_0$,
	\[
		\mbf P\pars{ \sup_{(x,t), (\tilde x,\tilde t) \in B_{1/2} \times [-1/2,0]} \frac{ |u(x,t) - u(\tilde x, \tilde t)|}{|x - \tilde x|^\alpha + |t - \tilde t|^{\alpha/(q - \alpha(q-1))}} > \lambda} \le \mbf P \pars{ (\mcl S - M)_+ > \lambda^{\sigma}} + \frac{C \nor{f}{\oo}^p}{\lambda^{\sigma p}}.
	\]
\end{theorem}
To state the assumptions for the regularity results for \eqref{E:main2}, we introduce the notation, for any $X \in \mbb S^d$,
\[
	m_+(X) := \max_{|v| \le 1} v \cdot Xv \quad \text{and} \quad m_-(X) := \min_{|v| \le 1} v \cdot Xv.
\]
That is, $m_+(X)$ and $m_-(X)$ are, respectively, the largest nonnegative and lowest nonpositive eigenvalue of $X$. Note that, if $F: \mbb S^d \to \RR$ is uniformly continuous and degenerate elliptic, then, for some constants $\nu > 0$ and $A > 0$ and for all $X \in \mbb S^d$,
\[
	- \nu m_+(X) - A \le F(X) \le - \nu m_-(X) + A.
\]
In order for the coercivity in the gradient to dominate the second-order dependence of $F$ at small scales, it is necessary to assume that the growth of $F$ in $Du$ is super-quadratic.

\begin{theorem}\label{T:introsecondorder}
	Assume that, for some $A > 1$, $q > 2$, $\nu > 0$, and $K > 0$,
	\begin{equation}\label{A:Fsuperquadratic}
		\left\{
		\begin{split}
		&- \nu m_+(X) + \frac{1}{A} |p|^q - A \le F(X,p,x,t) \le -\nu m_-(X) + A|p|^q + A\\
		&\text{for all } (X,p,x,t) \in \mbb S^d \times \RR^d \times \RR^d \times [-1,0],
		\end{split}
		\right.
	\end{equation}
	\begin{equation}\label{A:fC2bound}
		f \in C^2(\RR^d,\RR^m), \quad \nu + \nor{f}{\oo} + \nor{f}{\oo} \cdot \nor{Df}{\oo} + \nu \nor{f}{\oo} \nor{D^2 f}{\oo} \le K,
	\end{equation}
	and $u$ solves \eqref{E:main2} in $Q_1$ and satisfies \eqref{A:introrandombound}. Fix $M > 0$ and $p \ge 1$. Then there exist $\alpha = \alpha(A,q) > 0$, $\sigma = \sigma(A,q) > 0$, $\lambda_0 = \lambda_0(A,K,M,q) > 0$, and $C = C(A,K,M,p,q) > 0$ such that, for all $\lambda \ge \lambda_0$,
	\[
		\mbf P\pars{ \sup_{(x,t), (\tilde x,\tilde t) \in B_{1/2} \times [-1/2,0]} \frac{ |u(x,t) - u(\tilde x, \tilde t)|}{|x - \tilde x|^\alpha + |t - \tilde t|^{\alpha/(q - \alpha(q-1))}} > \lambda} \le \mbf P \pars{ (\mcl S - M)_+ > \lambda^{\sigma}} + \frac{C \nor{f}{\oo}^p}{\lambda^{\sigma p}}.
	\]
\end{theorem}

Although the bounds in Theorem \ref{T:introfirstorder} and \ref{T:introsecondorder} do depend on the regularity of $f$, the important point is that the dependence is scale-invariant. Indeed, the function $f^\eps$ defined by $f^\eps(x) := \eps^{1/2} f(x/\eps)$ satisfies
\[
	\nor{f^\eps}{\oo} = \eps^{1/2} \nor{f}{\oo}, \quad \nor{Df^\eps}{\oo} := \frac{1}{\eps^{1/2}} \nor{Df}{\oo}, \quad \text{and} \quad \nor{D^2 f^\eps}{\oo} = \frac{1}{\eps^{3/2}} \nor{D^2 f}{\oo}.
\]
As a consequence, $f^\eps$ satisfies \eqref{A:fC1bound} and \eqref{A:fC2bound} with some $K > 0$ independent of $\eps$ (the latter because, in \eqref{A:Fsuperquadratic}, $\nu$ is replaced with $\eps \nu$). This leads to the following scale-invariant estimates for the regularizing effect of \eqref{E:scaledmain1} and \eqref{E:scaledmain2}.
 
\begin{theorem}\label{T:introfirstorderscaling}
	For $A > 1$, $M > 0$, and $q > 1$, assume that
	\[
		\frac{1}{A}|p|^q - A \le H(p,x,t) \le A|p|^q + A
	\]
	and $f \in C^1_b(\RR^d,\RR^m)$, and, for $0 < \eps < 1$, let $u^\eps$ be the solution of \eqref{E:scaledmain1} with $\nor{u^\eps(\cdot,0)}{\oo} \le M$. Fix $\tau > 0$, $R > 0$, and $T > 0$. Then there exist $C = C(R,\tau,T,A, \nor{f}{C^1},M, q) > 0$, $\alpha = \alpha(A,q) > 0$, and $\sigma = \sigma(A,q) > 0$ such that, for all $\lambda > 0$,
	\[
		\mbf P \pars{ \sup_{(x,t), (\tilde x,\tilde t) \in B_R \times [\tau,T]} \frac{|u^\eps(x,t) - u^\eps(\tilde x, \tilde t)|}{|x - \tilde x|^\alpha + |t - \tilde t|^{\alpha/(q - \alpha(q-1))} } > C + \lambda} \le \frac{C\eps^{p/2}}{\lambda^{\sigma p}}.
	\]
\end{theorem}

\begin{theorem}\label{T:introsecondorderscaling}
	For $A > 1$, $\nu > 0$, $M > 0$, and $q > 2$, assume that
	\[
		-\nu m_+(X) + \frac{1}{A}|p|^q - A \le F(X,p,x,t) \le - \nu m_-(X) + A|p|^q + A
	\]
	and $f \in C^2_b(\RR^d,\RR^m)$, and, for $0 < \eps < 1$, let $u^\eps$ be the solution of \eqref{E:scaledmain2} with $\nor{u^\eps(\cdot,0)}{\oo} \le M$. Fix $\tau > 0$, $R > 0$, and $T > 0$. Then there exist $C = C(\nu,R,\tau,T,A, \nor{f}{C^2},M, q) > 0$, $\alpha = \alpha(A,q) > 0$, and $\sigma = \sigma(A,q) > 0$ such that, for all $\lambda > 0$,
	\[
		\mbf P \pars{ \sup_{(x,t), (\tilde x,\tilde t) \in B_R \times [\tau,T]} \frac{|u^\eps(x,t) - u^\eps(\tilde x, \tilde t)|}{|x - \tilde x|^\alpha + |t - \tilde t|^{\alpha/(q - \alpha(q-1))} } > C + \lambda} \le \frac{C\eps^{p/2}}{\lambda^{\sigma p}}.
	\]
\end{theorem}

A natural question is whether the methods and results of this paper can be generalized to treat a fixed, deterministic path $B$ that is, say, $\kappa$-H\"older continuous for some $\kappa \in (0,1)$. We strongly suspect that Theorems \ref{T:introfirstorder} and \ref{T:introsecondorder} can be adapted to such a setting in a straightforward manner and, in this case, the bounds in \eqref{A:fC1bound} and \eqref{A:fC2bound} are replaced by, respectively,
\[
	\nor{f}{\oo} + \nor{f}{\oo}^\kappa \nor{Df}{\oo}^{1-\kappa} \le K
	\quad \text{and} \quad
	\nu + \nor{f}{\oo} + \nor{f}{\oo}^\kappa \nor{Df}{\oo}^{1-\kappa} + \nor{f}{\oo}^\kappa (\nu \nor{D^2 f}{\oo})^{1-\kappa} \le K,	
\]
with constants depending additionally on the H\"older semi-norm of $B$. However, our Theorems \ref{T:introfirstorderscaling} and \ref{T:introsecondorderscaling} regarding Brownian motion do not immediately follow from such a statement. Indeed, with probability one, Brownian paths are $\kappa$-H\"older continuous if and only if $\kappa < 1/2$. The function $f^\eps(x) = \eps^{1/2} f(x/\eps)$, which arises due to the hyperbolic scaling in \eqref{E:scaledmain1} and \eqref{E:scaledmain2} as well as the self-similarity of Brownian motion, then satisfies
\[
	\nor{f^\eps}{\oo}^{\kappa} \nor{Df^\eps}{\oo}^{1-\kappa} = \eps^{-(1/2 - \kappa)} \nor{f}{\oo} \nor{Df}{\oo}.
\]
This quantity blows up as $\eps \to 0$ if $\kappa < 1/2$. We therefore emphasize that the methods used to prove Theorems \ref{T:introfirstorder} - \ref{T:introsecondorderscaling} are really probabilistic in nature, and use features of Brownian paths beyond their almost-sure regularity, in particular, the independence and scaling properties of increments (see Lemmas \ref{L:intcontrol1} and \ref{L:intcontrol2} in the appendix).

\subsection{Background}
The regularizing effects of Hamilton-Jacobi-Bellman equations like
\begin{equation}\label{E:classical}
	\del_t u + F(D^2 u, Du, x,t) = 0
\end{equation}
have been studied by many authors, including Cardaliaguet \cite{C}, Cannarsa and Cardaliaguet \cite{CC}, and Cardalia\-guet and Silvestre \cite{CS}, Chan and Vasseur \cite{chan2015giorgi} and Stockols and Vasseur \cite{stokols2017giorgi}. In these works, under a coercivity assumption on $F$ in the gradient variable  (but no regularity condition on $F$), bounded solutions are seen to be H\"older continuous, with estimate and exponents depending only on the growth of the $F$ in $Du$. These results were used to obtain homogenization results for problems set on periodic or stationary-ergodic spatio-temporal media; see, for instance, Schwab \cite{S} and Jing, Souganidis, and Tran \cite{JST}.

The equations \eqref{E:main1} and \eqref{E:main2} do not fit into this framework, due to the singular term on the right-hand side, which is nowhere pointwise-defined. A simple transformation (see Definition \ref{D:solutions} below) leads to a random equation that is everywhere pointwise-defined of the form \eqref{E:classical}. More precisely, if $u$ solves \eqref{E:main2} and 
\[
	\tilde u(x,t) = u(x,t) - f(x) \cdot B(t),
\]
then
\[
	\del_t \tilde u + F(D^2 \tilde u + D^2 f(x) \cdot B(t), D \tilde u + Df(x) \cdot B(t), x,t) = 0.
\]
However, this strategy does not immediately yield scale-invariant estimates. Indeed, the transformed equation corresponding to \eqref{E:scaledmain2} is, for $\eps > 0$,
\[
	\del_t \tilde u^\eps + F \pars{ \eps D^2 \tilde u^\eps + \frac{1}{\eps^{1/2}} D^2 f\pars{ \frac{x}{\eps}} B^\eps(t), D \tilde u^\eps + \frac{1}{\eps^{1/2}} Df \pars{ \frac{x}{\eps}} B^\eps(t), \frac{x}{\eps}, \frac{t}{\eps} } = 0,
\]
for which the results in the above references yield estimates that depend on $\eps$.

These issues were considered by Seeger \cite{Shomog} for the equation \eqref{E:main1} with $H$ independent of $(x,t)$ and convex in $p$. In this paper, we further extend the regularity results from \cite{Shomog} to apply also to second-order equations and with more complicated $(x,t)$-dependence for $F$ and $H$.  To do so, we follow \cite{CS} and prove that the equations exhibit an improvement of oscillation effect at all sufficiently small scales, which is a consequence only of the structure of the equation. The main difference with \cite{CS} is the addition of the random forcing term $f(x)\cdot dB(t)$ which obliges to revisit the analysis of \cite{CS} in a substantial way. 

\subsection{Organization of the paper}
In Section \ref{S:prelim}, we discuss the notion of pathwise viscosity solutions of equations like \eqref{E:main1} and \eqref{E:main2}, and we present a number of lemmas needed throughout the paper. The interior estimates are proved in Sections \ref{S:firstorder} and \ref{S:secondorder}, and the results for initial value problems are presented in Section \ref{S:apps}. Finally, in Appendix \ref{S:intcontrol}, we prove some results on controlling certain stochastic integrals.

\subsection{Notation}

 If $a$ and $b$ are real numbers, then we set $a\vee b=\max \{a, b\}$, $a\wedge b=\min\{a,b\}$ and  denote by $\lceil a \rceil$  the smallest integer greater than or equal to $a$. We let $\mbb S^d$ be the set of  symmetric real matrices of size $d\times d$. We say that a map $F: \mbb S^d\to \RR$ is degenerate elliptic if, for $X,Y\in \mbb S^d$ with $X\leq Y$, we have $F(X)\geq F(Y)$. Given $H: \RR^d \to \RR$, $H^*$ is defined for $\alpha \in \RR^d$ by $H^*(\alpha)=\sup_{p\in \RR^d} \left\{ \alpha\cdot p -H(p) \right\}$. Given a subset $C$ of $\RR^d$ and $-\oo< t_0< t_1<\oo$, $\partial^*(C\times (t_0,t_1))$ denotes the parabolic boundary of $C\times (t_0,t_1)$, namely
$$
\partial^*(C\times (t_0,t_1))= (C\times \{t_0\})\cup (\partial C\times (t_0,t_1)).
$$
For an open domain $U \subset \RR^N$, $USC(U)$ (respectively $LSC(U)$) denotes the space of upper- (respectively lower-) semicontinuous functions on $U$, and $BUC(U)$ is the space of bounded and uniformly continuous functions on $U$. For a bounded function $u: U \to \RR$, we define $\osc_U u := \sup_U u - \inf_U u$.

\section{Preliminaries}\label{S:prelim}
\subsection{Pathwise viscosity solutions}
Fix   $-\oo < t_0 < t_1 < \oo$ and let $U \subset \RR^d\times (t_0,t_1)$ be an open set. For $\zeta \in C((t_0,t_1), \RR^m)$, a degenerate elliptic    $F\in C(\mbb S^d \times \RR^d \times U\times (t_0,t_1),\RR)$, and $f \in C^2(\RR^d,\RR^m)$, we discuss the meaning of viscosity sub- and super-solutions of the equation
\begin{equation}\label{E:pathwise}
	du + F(D^2 u, Du, x, t)dt = f(x) \cdot d\zeta, \quad (x,t) \in U.
\end{equation}
The general theory of pathwise viscosity solutions, initiated by Lions and Souganidis \cite{LS1,LS2,LS3,LS4,Snotes}, covers a wide variety of equations for which $f$ may also depend on $u$ or $Du$. In the case of \eqref{E:pathwise}, the theory is much more tractable, and solutions are defined through a simple transformation.

\begin{definition}\label{D:solutions}
	A function $u \in USC(U)$ (resp. $u \in LSC(U)$) is a sub- (resp. super-) solution of \eqref{E:pathwise} if the function $\tilde u$ defined, for $(x,t) \in U$, by
	\[
		\tilde u(x,t) = u(x,t) - f(x) \cdot \zeta(t)
	\]
	is a sub- (resp. super-) solution of the equation
	\[
		\del_t \tilde u + F(D^2 \tilde u + D^2 f(x) \zeta(t), D \tilde u + Df(x) \zeta(t), x,t) = 0, \quad (x,t) \in U.
	\]
	A solution $u \in C(U)$ is both a sub- and super-solution.
\end{definition}

We remark that, if $F$ is independent of $D^2 u$, then we may take $f \in C^1(\RR^d,\RR^m)$.

We will often denote the fact that $u$ is a sub- (resp. super-) solution of \eqref{E:pathwise}, by writing
\[
	du + F(D^2 u, Du, x,t) dt \le f(x) \cdot d\zeta \quad \pars{ \text{resp. } du + F(D^2 u, Du, x,t) dt \ge f(x) \cdot d\zeta }.
\]
At times, when it does not cause confusion, we also use the notation
\[
	\del_t u + F(D^2 u, Du, x, t) = f(x) \cdot \dot\zeta(t),
\]
even when $\zeta$ is not continuously differentiable. This will become particularly useful in proofs that involve scaling, in which case the argument of $\dot \zeta$ may change.

\subsection{Control and differential games formulae}\label{SS:formulae}

Just as for classical viscosity solutions, some equations allow for representation formulae with the use of the theories of optimal control or differential games. Before we explain this, we give meaning to certain pathwise integrals that come up in the formulae.

\begin{lemma}\label{L:intzeta1}
	Assume that $s < t$ and $f \in C^{0,1}([s,t],\RR^m)$. Then the map
	\[
		C^1([s,t],\RR^m) \ni \zeta \mapsto \int_s^t f(r)\cdot \dot \zeta(r)dr = \sum_{i=1}^m \int_s^t f^i(r) \cdot \dot \zeta^i(r)dr
	\]
	extends continuously to $\zeta \in C([s,t],\RR^m)$.
\end{lemma}

\begin{proof}
	The result is immediate upon integrating by parts, which yields, for $\zeta \in C^1([s,t],\RR^m)$,
	\[
		 \int_s^t f(r)\dot \zeta(r)dr = f(t) \zeta(t) - f(s) \zeta(s) - \int_s^t \dot f(r) \zeta(r)dr.
	\]
\end{proof}

\begin{lemma}\label{L:intzeta2}
	Assume that $s < t$, $f \in C^1_b(\RR^d,\RR^m)$, $W: [s,t] \times \mcl A \to \RR$ is a Brownian motion on some probability space $(\mcl A, \mcl F, \mbb P)$, $\alpha,\sigma: [s,t] \times \mcl A \to \RR^d$ are  bounded and progressively measurable with respect to the filtration of $W$, $\tau \in [s,t]$ is a $W$-stopping time, and
	\[
		dX_r = \alpha_r dr + \sigma_r dW \quad \text{for } r \in [s,t].
	\]
	Then the map
	\[
		C^1([s,t],\RR^m) \ni \zeta \mapsto \int_s^\tau f(X_r)\cdot \dot \zeta(r)dr = \sum_{i=1}^m \int_s^\tau f^i(X_r) \cdot \dot \zeta^i(r)dr \in L^2(\mcl A)
	\]
	extends continuously to $\zeta \in C([s,t],\RR^m)$, and, moreover,
	\begin{align*}
		\mbb E \left[ \int_s^\tau f(X_r)\cdot \dot \zeta(r)dr \right] &= \mbb E\left[ f(X_\tau)\cdot \zeta(\tau) - f(X_s) \cdot \zeta(s) \right] \\
		&- \mbb E\left[ \int_s^\tau \zeta(r) \cdot \pars{ Df(X_r) \cdot \alpha_r  + \frac{1}{2} \langle D^2 f(X_r) \sigma_r, \sigma_r \rangle }dr \right].
	\end{align*}
\end{lemma}

\begin{proof}
	If $\zeta \in C^1([s,t],\RR^m)$, then It\^o's formula yields, for $i = 1,2,\ldots,m$,
	\begin{align*}
		d \left[ f^i(X_r)\cdot \zeta^i(r) \right] &= \left[ f^i(X_r) \dot \zeta^i(r) + Df^i(X_r) \cdot \alpha_r \zeta^i(r) + \frac{1}{2} \langle D^2 f^i(X_r) \sigma_r, \sigma_r \rangle \zeta^i(r) \right]dr \\
		&+ (Df^i(X_r) \cdot \sigma_r \zeta^i(r))dW_r,
	\end{align*}
	and so
	\begin{equation}\label{Ito}
		\begin{split}
		\int_s^\tau f^i(X_r) \dot \zeta^i(r)dr &= f^i(X_\tau)\zeta^i(\tau) - f^i(X_s) \zeta^i(s) - \int_s^\tau \zeta^i(r) \pars{ Df^i(X_r) \cdot \alpha_r  + \frac{1}{2} \langle D^2 f^i(X_r) \sigma_r, \sigma_r \rangle}dr \\
		&- \int_s^\tau \zeta^i(r) Df(X_r) \cdot \sigma_r dW_r.
		\end{split}
	\end{equation}
	The It\^o isometry property implies that
	\[
		L^2([s,t]) \ni \zeta^i \mapsto \int_s^\tau \zeta^i(r) Df(X_r) \cdot \sigma_r dW_r \in L^2(\mcl A)
	\]
	is continuous, and, in particular, the map extends to $\zeta^i \in C([s,t])$. The result follows from the fact that the other terms on the right-hand side of \eqref{Ito} are continuous with respect to $\zeta^i \in C([s,t])$. The final claim follows upon taking the expectation of both sides of \eqref{Ito} and appealing to the optional stopping theorem.
\end{proof}

For arbitrary continuous $\zeta$, we freely interchange notations such as
\[
	\int_s^t f_r \cdot d\zeta_r \quad \text{and} \quad \int_s^t f(r) \cdot \dot \zeta(r) dr.
\]
Throughout the paper, $\zeta$ is often taken to be a Brownian motion, defined on a probability space that is independent of $W$.

We now consider some equations for which sub- and super-solutions can be compared from above or below with particular formulae. For convenience, we write the equations backward in time.
%
%Recall that, for an open set $\mcl C \subset \RR^d$ and $t_1 < t_0$, we denote the parabolic boundary of $\mcl C \times (t_1,t_0)$ as
%\[
%	\del^* (\mcl C \times (t_1,t_0)) = (\del \mcl C \times [t_1,t_0] ) \cup (\mcl C \times \{t_1\} ).
%\]
%For $x_0 \in \mcl C$, we set
%\[
%	\mcl A_{\mcl C \times (t_1,t_0), x_0} := \left\{ (s,\gamma) \in [t_1,t_0] \times W^{1,\oo}([s,t_0],\RR^d) : (s,\gamma(s)) \in \del^* (\mcl C \times (t_1,t_0)), \; \gamma(t_0) = x_0 \right\}.
%\]

\begin{lemma}\label{L:HJformula}
	Assume $\mcl C \subset \RR^d$ is open, $x_0 \in \mcl C$, $t_0 < t_1$, $U$ is an open domain containing $\oline{\mcl C} \times [t_1,t_0]$, $\zeta \in C(\RR,\RR^m)$, $f \in C^1(U)$, and $H: \RR^d \to \RR$ is convex and superlinear. Let $u \in C(U)$ be a pathwise viscosity sub- (resp. super-) solution, in the sense of Definition \ref{D:solutions}, of
	\[
		-du + H(Du)dt = f(x) \cdot d\zeta \quad \text{in } U.
	\]
	Then
	\[
		u(x_0,t_0) \le \; \pars{ \text{resp.} \ge } \; \inf \left\{ u(\gamma_\tau,\tau) + \int_{t_0}^{\tau} H^*(-\dot\gamma_r) dr + \int_{t_0}^{\tau} f(\gamma_r) \cdot d\zeta_r : \gamma \in W^{1,\oo}([t_0,t_1], \RR^d), \; \gamma_{t_0} = x_0  \right\},
	\]
	where, for fixed $\gamma \in W^{1,\oo}([t_0,t_1], \RR^d)$,
	\begin{equation}\label{stoppingtime1}
		\tau = \tau^\gamma := \inf\{ t \in (t_0,t_1] : \gamma_t \in \del \mcl C \}.
	\end{equation}
\end{lemma}

\begin{proof}
	We prove the claim for sub-solutions, as it is identical for super-solutions. 
	
	Definition \ref{D:solutions} implies that if
	\[
		\tilde u(x,t) := u(x,t) + f(x) \cdot \zeta(t) \quad \text{for } (x,t) \in U,
	\]
	then $\tilde u$ is a sub-solution of the boundary-terminal-value problem
	\begin{equation}\label{E:BTVP}
		\begin{dcases}
		-\del_t \tilde u+ H(D \tilde u - Df(x) \cdot \zeta(t)) = 0 & \text{in } \mcl C \times [t_0,t_1) \text{ and}\\
		\tilde u(x,t) = u(x,t) + f(x) \cdot \zeta(t) & \text{if } t = t_1 \text{ or } x \in \del \mcl C.
		\end{dcases}
	\end{equation}
	The unique solution of \eqref{E:BTVP} (see \cite{L}) is given by
	\[
		w(x,t) = \inf \left\{ u(\gamma_\tau,\tau) + f(\gamma_\tau) \cdot \zeta(\tau) + \int_t^\tau \left[ H^*(-\dot \gamma_r) - \dot \gamma_r \cdot Df(\gamma_r) \cdot \zeta(r) \right]dr : \gamma \in W^{1,\oo}([t,t_1],\RR^d), \; \gamma_t = x\right\},
	\]
	where $\tau$ is as in \eqref{stoppingtime1}. Integrating by parts gives
	\[
		\int_t^\tau \dot \gamma_r \cdot Df(\gamma_r) \cdot \zeta(r)dr = f(\gamma_\tau) \zeta(\tau) - f(x) \zeta(t) - \int_t^\tau f(\gamma_r)\cdot d\zeta(r),
	\]
	and, hence,
	\[
		w(x,t) = f(x) \zeta(t) + \inf \left\{ u(\gamma_\tau,\tau) + \int_t^\tau H^*(-\dot \gamma_r) dr + \int_t^\tau f(\gamma_r) \cdot d\zeta(r) : \gamma \in W^{1,\oo}([t,t_1],\RR^d), \; \gamma_t = x \right\}.
	\]
	The result now follows because, by the comparison principle for \eqref{E:BTVP}, $\tilde u \le w$ on $\oline{\mcl C} \times [t_0,t_1]$.
\end{proof}

We next give formulae for solutions of some Hamilton-Jacobi-Bellman and Hamilton-Jacobi-Isaacs equations. 

For $-\oo < t_0 < t_1 < \oo$, assume that
\begin{equation}\label{W}
	W: [t_0,t_1] \times \mcl A \to \RR \text{ is a Brownian motion defined on a probability space } (\mcl A, \mcl F, \mbb P),
\end{equation}
 with associated expectation $\mbb E$, and define the spaces of admissible controls
\begin{align*}
	&\mathscr C := \left\{ \mu \in L^\oo\pars{ [t_0,t_1] \times \mcl A,\RR^d } : \mu \text{ is adapted with respect to $W$}\right\} \text{ and}\\
	&\mathscr C_M := \left\{ \mu \in \mathscr C : \nor{\mu}{\oo} \le M \right\}.
\end{align*}
The Isaacs' equations require us to use the spaces of strategies defined by
\begin{align*}
	&\mathscr S := \left\{ \beta: \mathscr C \to \mathscr C : \mu_1 = \mu_2 \text{ on } [t_0,t] \; \Rightarrow \; \beta(\mu_1)(t) = \beta(\mu_2)(t) \right\} \text{ and}\\
	&\mathscr S_M := \left\{ \beta \in \mathscr S : \beta(\mathscr C) \subset \mathscr C_M \right\}.
\end{align*}

\begin{lemma}\label{L:HJBformula}
	Assume $\mcl C \subset \RR^d$ is open and convex, $x_0 \in \mcl C$, $t_0 < t_1$, $U$ is an open domain containing $\oline{\mcl C} \times [t_0,t_1]$, $f \in C^2(U)$, $H: \RR^d \to \RR$ is convex and superlinear, and $\nu > 0$. Given $(\alpha,\sigma) \in \mathscr C \times \mathscr C$, denote by $X = X^{\alpha,\sigma,x_0,t_0}$ the solution of
	\begin{equation}\label{SDEformula}
		dX_r = \alpha_r dr + \sigma_r dW_r \quad \text{in } [t_0,t_1] \text{ and } X_{t_0} = x_0,
	\end{equation}
	and
	\begin{equation}\label{stoppingtime2}
		\tau = \tau^{\alpha,\sigma,x_0,t_0} := \inf\left\{ t \in (t_0,t_1] : X^{\alpha,\sigma,x_0,t_0}_t \in \del \mcl C \right\}.
	\end{equation}
	
	\begin{enumerate}[(a)]
	\item Let $u \in C(U)$ be a pathwise viscosity super-solution, in the sense of Definition \ref{D:solutions}, of
	\[
		-du + \left[ - \nu m_-(D^2 u) + H(Du) \right]dt = f(x) \cdot d\zeta \quad \text{in }U.
	\]
	Then
	\[
		u(x_0,t_0) \ge \inf_{(\alpha,\sigma) \in \mathscr C \times \mathscr C_{\sqrt{2\nu}}} \mbb E \left[ u(X_\tau,\tau) + \int_{t_0}^\tau H^*(-\alpha_r)dr + \int_{t_0}^\tau f(X_r) \cdot d\zeta_r  \right].
	\]
	
	\item Let $u \in C(U)$ be a pathwise viscosity sub-solution, in the sense of Definition \ref{D:solutions}, of
	\[
		-du + \left[ - \nu m_+(D^2 u) + H(Du) \right]dt = f(x) \cdot d\zeta \quad \text{in } U.
	\]
	Then
	\[
		u(x_0,t_0) \le \inf_{\alpha \in \mathscr C} \sup_{\beta \in \mathscr S_{\sqrt{2\nu}}} \mbb E \left[ u(X_\tau,\tau) + \int_{t_0}^\tau H^*(-\alpha_r)dr + \int_{t_0}^\tau f(X_r) \cdot d\zeta_r \right],
	\]
	where $X$ and $\tau$ are as in respectively \eqref{SDEformula} and \eqref{stoppingtime2} with $\sigma = \beta(\alpha)$.
	\end{enumerate}
\end{lemma}

\begin{proof}
	As a preliminary step, assume that $(\alpha,\sigma) \in \mathscr C \times \mathscr C$ and $X$ and $\tau$ are as in \eqref{SDEformula} and \eqref{stoppingtime2}. Then Lemma \ref{L:intzeta2} gives
	\begin{equation}\label{intbypartsstep}
		\begin{split}
		\mbb E \left[ \int_t^\tau f(X_r) \cdot d\zeta_r \right] &= \mbb E\left[ f(X_\tau) \zeta(\tau) - f(X_t) \zeta(t)\right] \\
		&- \mbb E \left[ \int_t^\tau \zeta(r) \cdot \pars{ Df(X_r)\cdot \alpha_r + \frac{1}{2} D^2 f(X_r)\sigma_r \cdot \sigma_r } dr \right].
		\end{split}
	\end{equation}

	(a) By Definition \ref{D:solutions}, if
	\[
		\tilde u(x,t) := u(x,t) + f(x) \cdot \zeta(t),
	\]
	then $\tilde u$ is a classical viscosity super-solution of
	\begin{equation}\label{E:BTVP2}
		\begin{dcases}
		-\del_t \tilde u - \nu m_-\pars{ D^2 \tilde u - D^2 f(x) \cdot \zeta(t)} + H\pars{ D\tilde u - Df(x) \cdot \zeta(t)} = 0 & \text{in } \mcl C \times [t_0,t_1),\\
		\tilde u(x,t) = u(x,t) + f(x) \cdot \zeta(t) \quad \text{if } t = t_1 \text{ or } x \in \del \mcl C.
		\end{dcases}
	\end{equation}
	For $(X,p,x,t) \in \mbb S^d \times \RR^d \times U$, we have
	\begin{align*}
		&-\nu m_-\pars{ X - D^2 f(x) \cdot \zeta(t)} + H\pars{ p - Df(x) \cdot \zeta(t)}\\
		&=
		\sup_{|\sigma| \le \sqrt{2\nu}, \; \alpha \in \RR^d} \left\{ - \frac{1}{2} \sigma \cdot X\sigma + \frac{1}{2} \sigma \cdot D^2 f(x) \sigma \cdot \zeta(t) - \alpha \cdot p + \alpha \cdot Df(x) \cdot \zeta(t) - H^*(-\alpha) \right\},
	\end{align*}
	and so standard results from the theory of stochastic control (see Theorem II.3 in \cite{L}) imply that the unique solution of \eqref{E:BTVP2} is given by
	\begin{align*}
		w(x,t) &:= \inf_{(\alpha,\sigma) \in \mathscr C \times \mathscr C_{\sqrt{2\nu}}} 
		\mbb E \bigg[ u(X_\tau,\tau) + f(X_\tau) \cdot \zeta(\tau)\\
		&+ \int_t^\tau \left[ H^*(-\alpha_r) -\zeta(r) \cdot \pars{ \alpha_r \cdot Df(X_r) + \frac{1}{2} \sigma_r \cdot D^2 f(X_r) \sigma_r }\right]dr \bigg]\\
		&= f(x) \cdot \zeta(t) + \inf_{(\alpha,\sigma) \in \mathscr C \times \mathscr C_{\sqrt{2\nu}}} \mbb E \left[ u(X_\tau,\tau) + \int_t^\tau H^*(-\alpha_r) dr + \int_t^\tau f(X_r) \cdot d\zeta_r \right],
	\end{align*}
	where the last equality follows from \eqref{intbypartsstep}. The result follows from the comparison principle for \eqref{E:BTVP2}, which implies that $\tilde u(x,t) \ge w(x,t)$ for $(x,t) \in \oline{\mcl C} \times [t_0,t_1]$.
	
	(b) By Definition \ref{D:solutions}, if
	\[
		\tilde u(x,t) := u(x,t) + f(x) \cdot \zeta(t),
	\]
	then $\tilde u$ is a classical viscosity sub-solution of
	\begin{equation}\label{E:BTVP3}
		\begin{dcases}		
		-\del_t\tilde u - \nu m_+\pars{ D^2 \tilde u - D^2 f(x) \cdot \zeta(t)} + H\pars{ D\tilde u - Df(x) \cdot \zeta(t)} = 0 & \text{in } \mcl C \times [t_0,t_1),\\
		\tilde u(x,t) = u(x,t) + f(x) \cdot \zeta(t) \quad \text{if } t = t_1 \text{ or } x \in \del \mcl C.
		\end{dcases}
	\end{equation}
	For $(X,p,x,t) \in \mbb S^d \times \RR^d \times U$, we have
	\begin{align*}
		&-\nu m_+\pars{ X - D^2 f(x) \cdot \zeta(t)} + H\pars{ p - Df(x) \cdot \zeta(t)}\\
		&= \sup_{\alpha \in \RR^d} \inf_{|\sigma| \le \sqrt{2\nu}} \left\{ - \frac{1}{2} \sigma \cdot X\sigma + \frac{1}{2} \sigma \cdot D^2 f(x) \sigma \cdot \zeta(t) - \alpha \cdot p + \alpha \cdot Df(x) \cdot \zeta(t) - H^*(-\alpha) \right\}\\
		&= \inf_{|\sigma| \le \sqrt{2\nu}} \sup_{\alpha \in \RR^d}\left\{ - \frac{1}{2} \sigma \cdot X\sigma + \frac{1}{2} \sigma \cdot D^2 f(x) \sigma \cdot \zeta(t) - \alpha \cdot p + \alpha \cdot Df(x) \cdot \zeta(t) - H^*(-\alpha) \right\},
	\end{align*}
	and so standard results from the theory of stochastic differential games (see Theorem 2.6 of \cite{FS}) imply that, keeping in mind that $\sigma = \beta(\alpha)$ below, the unique solution of \eqref{E:BTVP3} is given by
	\begin{align*}
		w(x,t) &:= \inf_{\alpha \in \mathscr C} \sup_{\beta \in \mathscr S_{\sqrt{2\nu}}}
		\mbb E \bigg[ u(X_\tau,\tau) + f(X_\tau) \cdot \zeta(\tau)\\
		&+ \int_t^\tau \left[ H^*(-\alpha_r) -\zeta(r) \cdot \pars{ \alpha_r \cdot Df(X_r) + \frac{1}{2} \sigma_r \cdot D^2 f(X_r) \sigma_r }\right]dr
		\bigg]\\
		&= f(x) \cdot \zeta(t) + \inf_{\alpha \in \mathscr C} \sup_{\beta \in \mathscr S_{\sqrt{2\nu}}} \mbb E_{x,t} \left[ u(X_\tau,\tau) + \int_t^\tau H^*(-\alpha_r) dr + \int_t^\tau f(X_r) \cdot d\zeta_r \right],
	\end{align*}
	where \eqref{intbypartsstep} gives the last equality. The result follows from the comparison principle for \eqref{E:BTVP2}, which implies that $\tilde u(x,t) \le w(x,t)$ for $(x,t) \in \oline{\mcl C} \times [t_0,t_1]$.
\end{proof}

\subsection{Comparison with homogenous equations} We now take $\zeta$ to be a Brownian motion, and we assume that
\begin{equation}\label{BM}
	B: [-1,0] \times \Omega \to \RR^m \quad \text{is a standard Brownian motion on the probability space } (\Omega, \mbf F, \mbf P).
\end{equation}
In this case, the forcing term $\sum_{i=1}^m f^i(x) \cdot dB^i(t)$ is nowhere pointwise defined, and the naive estimate
\[
 	\abs{ \sum_{i=1}^m f^i(x) \cdot dB^i(t)} \le \nor{f}{\oo} \nor{dB}{\oo}
\]
cannot be used in comparison principle arguments, as would be the case if $B$ belonged to $C^1$.

The results given below provide another way to compare solutions of \eqref{E:main1} and \eqref{E:main2} with equation that are independent of $x$ and $t$. In the new equations, the forcing term is replaced with a random constant that depends on $f$ only through quantities as in \eqref{A:fC1bound} and \eqref{A:fC2bound}, at the expense of slightly weakening the coercivity bounds in the gradient variable. The main tool is to use Lemmas \ref{L:intcontrol1} and \ref{L:intcontrol2} to control the stochastic integrals that arise from the representation formulae in Lemmas \ref{L:HJformula} and \ref{L:HJBformula}.

For $q > 1$, define
\[
	q' := \frac{q}{q-1} \quad \text{and} \quad c_q := (q-1) q^{-q/(q-1)},
\]
so that, in particular, for any constant $a > 0$, the convex conjugate of $p \mapsto a |p|^q$ is given by
\begin{equation}\label{convexdual}
	\pars{a |\cdot|^q }^* = c_q a^{-(q'-1)} |\cdot|^{q'}.
\end{equation}

\begin{lemma}\label{L:Dbarrier}
	Let $B$ be as in \eqref{BM} and fix $m > 0$, $K > 0$, $q > 1$, and $\kappa \in (0,1/2)$. Then there exist a random variable $\mcl D: \Omega \to \RR_+$ and $\lambda_0 = \lambda_0(\kappa, m,K,q) > 0$ such that the following hold:
	\begin{enumerate}[(a)]
	\item For any $p \ge 1$, there exists a constant  $C = C(\kappa,K,p,q) > 0$ such that, for all $\lambda \ge \lambda_0$,
	\[
		\mbf P( \mcl D > \lambda) \le \frac{Cm^p}{\lambda^p}.
	\]
	\item Let $f \in C^1(\RR^d,\RR^m)$ satisfy
	\[
		\nor{f}{\oo} \le m \quad \text{and} \quad  \nor{f}{\oo}(1+ \nor{Df}{\oo}) \le K,
	\]
	and assume that $A > 1$, $\eps_1,\eps_2: \Omega \to (0,1)$, and $-1 + \eps_2 \le r_0 \le 0$. Suppose that, for some $R \in (0,\oo]$, $w$ solves 
	\[
		\begin{dcases}
			\del_tw + \frac{1}{A} |Dw|^q - \pars{ \frac{\eps_2}{\eps_1}}^{q'} A \le \pars{ \frac{\eps_2}{\eps_1}}^{q'}  f(\eps_1 x) \cdot \dot B(r_0 + \eps_2 t) & \text{and}\\
			\del_t w+ A|Dw|^q + \pars{ \frac{\eps_2}{\eps_1}}^{q'} A \ge \pars{ \frac{\eps_2}{\eps_1}}^{q'}  f(\eps_1 x) \cdot \dot B(r_0 + \eps_2 t) & \text{in } B_R \times [-1,0],
		\end{dcases}
	\]
	fix an open convex set $\mcl C \subset B_R$, $x_0 \in \mcl C$, and $-1 \le t_1 < t_0 \le 0$. Then
	\[
		w_-(x_0,t_0) - \frac{\eps_2^{q'-1 + \kappa}}{\eps_1^{q'}} A\mcl D
		\le w(x_0,t_0) \le w_+(x_0,t_0) + \frac{\eps_2^{q'-1 + \kappa}}{\eps_1^{q'}} A\mcl D,
	\]
	where
	\[
		\begin{dcases}
			\del_t w_{-} + 2A |Dw_-|^q = 0 & \text{and} \\
			\del_t w_{+} + \frac{1}{2A}|Dw_+|^q = 0 & \text{in } \mcl C \times (t_1,t_0], \text{ and}\\
			w_- = w_+ = w & \text{on } \del^*(\mcl C \times (t_1,t_0)).
		\end{dcases}
	\]
	\end{enumerate}
\end{lemma}

\begin{proof}
{\it Step 1.} For $(x,t) \in \oline{B_R} \times [0,1]$, define $\tilde w(x,t) := w(x,-t)$ and $\tilde B(t) := B(0) - B(-t)$. Then $\tilde B: [0,1] \times \Omega \to \RR^m$ is a Brownian motion, and $\tilde w$ solves
	\[
		\begin{dcases}
			-\del_t\tilde w 
			+ \frac{1}{A} |D\tilde w|^q - \pars{ \frac{\eps_2}{\eps_1}}^{q'}  A \le \pars{ \frac{\eps_2}{\eps_1}}^{q'} f(\eps_1 x) \cdot \dot {\tilde B}(-r_0 + \eps_2 t) & \text{and}\\
			-\del_t\tilde w
			+ A|D\tilde w|^q + \pars{ \frac{\eps_2}{\eps_1}}^{q'} A \ge \pars{ \frac{\eps_2}{\eps_1}}^{q'} f(\eps_1 x) \cdot \dot {\tilde B}(-r_0 + \eps_2 t) & \text{in } B_R \times [0,1].
		\end{dcases}
	\]
	We also define $\tilde w_+(x,t) = w_+(x,-t)$ and $\tilde w_-(x,t) = w_-(x,-t)$, which solve
	\[
		\begin{dcases}
			-\del_t\tilde w_{-} + 2A |Dw_-|^q = 0 & \text{and} \\
			- \del_t\tilde w_{+} + \frac{1}{2A}|Dw_+|^q = 0 & \text{in } \mcl C \times [-t_0,-t_1), \text{ and}\\
			\tilde w_- = \tilde w_+ = \tilde w & \text{on } (\mcl C \times \{-t_1\}) \cup (\del \mcl C \times [-t_0,-t_1]).
		\end{dcases}
	\]
	The classical Hopf-Lax formula and \eqref{convexdual} then give, for $(x,t) \in \oline{\mcl C} \times [-t_0,-t_1]$,
	\begin{align*}
		\tilde w_+(x,t) &= \inf_{(y,s) \in (\mcl C \times \{-t_1\}) \cup (\del \mcl C \times [-t_0,-t_1])} \left\{ \tilde w(y,s) + c_q(2A)^{q'-1} \frac{|x-y|^{q'}}{|t-s|^{q'-1}} \right\} \text{ and}\\
		\tilde w_-(x,t) &= \inf_{(y,s) \in (\mcl C \times \{-t_1\}) \cup (\del \mcl C \times [-t_0,-t_1])} \left\{ \tilde w(y,s) + c_q(2A)^{-(q'-1)} \frac{|x-y|^{q'}}{|t-s|^{q'-1}} \right\}.
	\end{align*}
	
	{\it Step 2.} Let $\kappa \in (0,1/2)$ and $\mcl D$ be as in Lemma \ref{L:intcontrol1}. Then, by that lemma, for any $0 < \delta < 1$, $\gamma \in W^{1,\oo}([-t_0,-t_1],\RR^d)$, and $\tau \in [-t_0,-t_1]$,
	\begin{equation}\label{intcontroluse}
	\begin{split}
		\pars{ \frac{\eps_2}{\eps_1}}^{q'}& \abs{ \int_{-t_0}^{\tau} f(\eps_1 \gamma_r) \cdot \dot{\tilde B}(-r_0 + \eps_2 r)dr}\\
		&= \frac{\eps_2^{q'-1}}{\eps_1^{q'}} \abs{ \int_{-r_0 - \eps_2 t_0}^{-r_0 + \eps_2 \tau} f \pars{ \eps_1 \gamma \pars{ \frac{r +r_0}{\eps_2} } } \cdot \dot{\tilde B}(r)dr }\\
		&\le \frac{\eps_2^{q'-1}}{\eps_1^{q'}} \delta^{q'} \int_{-r_0 - \eps_2 t_0}^{-r_0 + \eps_2 \tau} \abs{ \frac{\eps_1}{\eps_2} \dot \gamma \pars{ \frac{r +r_0}{\eps_2}} }^{q'} dr + \frac{\eps_2^{q'-1 + \kappa}}{\eps_1^{q'}} \frac{\mcl D}{\delta^q} (\tau + t_0)^\kappa \\
		&= \delta^{q'} \int_{-t_0}^\tau |\dot \gamma_r|^{q'}dr + \frac{\eps_2^{q'-1 + \kappa}}{\eps_1^{q'}} \frac{\mcl D}{\delta^q} (\tau + t_0)^\kappa.
	\end{split}
	\end{equation}
	
	{\it Step 3.} We prove the upper bound first. By Lemma \ref{L:HJformula} and the equality \eqref{convexdual}, we have, with probability one,
	\begin{align*}
		\tilde w(x_0,-t_0) &\le \inf\left\{ \tilde w(\gamma_\tau,\tau) + c_q A^{q'-1} \int_{-t_0}^{\tau} |\dot \gamma_r|^{q'}dr + \pars{ \frac{\eps_2}{\eps_1}}^{q'} A(\tau + t_0) \right. \\
		&+ \left. \pars{ \frac{\eps_2}{\eps_1}}^{q'} \int_{-t_0}^\tau f(\eps_1 \gamma_r) \cdot \dot{\tilde B}(-r_0 + \eps_2 r)dr : \gamma \in W^{1,\oo}([-t_0,-t_1],\RR^d) \right\},
	\end{align*}
	where, as in \eqref{stoppingtime1}, we define
	\[
		\tau = \tau^\gamma := \inf \left\{ t \in (-t_0,-t_1] : \gamma_\tau \in \del \mcl C \right\}.
	\]
	We then set
	\[
		\delta = 1 \wedge \left[ (2^{q'-1} - 1)^{1/q'} c_q^{1/q'} A^{1/q} \right],
	\]
	which, in particular, implies that $\delta^{q'} \le c_q (2^{q'-1} - 1) A^{q'-1}$. Then, in view of \eqref{intcontroluse}, for some constant $C_q > 0$,
	\begin{equation}\label{tildewupper}
	\begin{split}
		\tilde w(x_0,-t_0) &\le \inf \left\{ \tilde w(\gamma_\tau,\tau) + c_q (2A)^{q'-1} \int_{-t_0}^{\tau} |\dot \gamma_r|^{q'} dr 
		 : \gamma \in W^{1,\oo}([-t_0,-t_1],\RR^d) \right\} \\
		 &\quad +  A \pars{ \frac{\eps_2}{\eps_1}}^{q'} \left[ 1 +  \frac{1}{\delta^q} \eps_2^{-(1-\kappa)}\mcl D \right] \\
		&\le \tilde w_+(x_0,-t_0) +   A \frac{\eps_2^{q' - 1 + \kappa}}{\eps_1^{q'}} (1 +  C_q\mcl D ).
	\end{split}
	\end{equation}
	
	{\it Step 4.} We next consider the lower bound. We again use \eqref{convexdual} and Lemma \ref{L:HJformula} to obtain
	\begin{align*}
		\tilde w(x_0,-t_0) &\ge \inf\left\{ \tilde w(\gamma_\tau,\tau) + c_q A^{-(q'-1)} \int_{-t_0}^\tau |\dot \gamma_r|^{q'}dr - \pars{ \frac{\eps_2}{\eps_1}}^{q'} A (\tau + t_0) \right. \\
		&+ \left.\pars{ \frac{\eps_2}{\eps_1}}^{q'}\int_{-t_0}^{\tau} f(\eps_1 \dot \gamma_r) \cdot \dot{\tilde B}(-r_0 + \eps_2 r)dr : \gamma \in W^{1,\oo}([-t_0,-t_1],\RR^d) \right\}.
	\end{align*}
	Choosing
	\[
		\delta := 1 \wedge c_q^{1/q'}(1 - 2^{-(q'-1)})^{1/q'} A^{-1/q}
	\]
	yields $\delta^{q'} \le c_q (1 - 2^{-(q'-1)}) A^{-(q'-1)}$. As a consequence, Jensen's inequality and \eqref{intcontroluse} yield, for some $C_q' > 0$,
	\begin{equation}\label{tildewlower}
	\begin{split}
		\tilde w(x_0,-t_0)
		&\ge \inf \left\{ \tilde w(\gamma_\tau,\tau) + c_q (2A)^{-(q'-1)} \int_{-t_0}^{\tau} |\dot \gamma_r|^{q'}dr : \gamma \in W^{1,\oo}([-t_0,-t_1]) \right\} \\
		& - \frac{\eps_2^{q'-1 + \kappa}}{\eps_1^{q'}} A( 1 + C_q'\mcl D )\\
		&\ge w_-(x_0,-t_0) - \frac{\eps_2^{q'-1 + \kappa}}{\eps_1^{q'}} A( 1 + C_q'\mcl D ).
	\end{split}
	\end{equation}
	
	{\it Step 5.} We set $\tilde{\mcl D} := 1 + (C_q \vee C_q')\mcl D$, so that, after performing a time change, \eqref{tildewupper} and \eqref{tildewlower} lead to
	\[
		w_-(x_0,t_0) - \frac{\eps_2^{q'-1 + \kappa}}{\eps_1^{q'}} A \tilde{\mcl D}
		\le w(x_0,t_0)
		\le w_+(x_0,t_0) +   A \frac{\eps_2^{q' - 1 + \kappa}}{\eps_1^{q'}} \tilde{\mcl D}.
	\]
	Let $\lambda_0$ be as in Lemma \ref{L:intcontrol1}. Then, for all
	\[
		\lambda \ge \widetilde{\lambda_0} := (1 + (C_q \vee C_q') \lambda_0) \vee 2,
	\]
	we have, for $C = C(\kappa, m,K,p,q) > 0$ as in Lemma \ref{L:intcontrol1},
	\[
		\mbf P(\tilde{\mcl D} > \lambda) = \mbf P \pars{ \mcl D > \frac{\lambda-1}{C_q \vee C_q'} }
		\le \frac{C (C_q \vee C_q')^p}{(\lambda-1)^p} \le \frac{2^p C (C_q \vee C_q')^p}{\lambda^p}.
	\]
\end{proof}

\begin{lemma}\label{L:Ebarrier}
	Let $B$ be as in \eqref{BM}, and fix $m > 0$, $K > 0$, $q > 1$, $\nu > 0$, and $\kappa \in (0,1/2)$. Then there exist a random variable $\mcl E: \Omega \to \RR_+$ and $\lambda_0 = \lambda_0(\kappa, m,K,q) > 0$ such that the following hold:
	\begin{enumerate}[(a)]
	\item For any $p \ge 1$, there exists a constant  $C = C(\kappa,K,p,q) > 0$ such that, for all $\lambda \ge \lambda_0$,
	\[
		\mbf P( \mcl E > \lambda) \le \frac{Cm^p}{\lambda^p}.
	\]
	\item Let $f \in C^2(\RR^d,\RR^m)$ satisfy
	\[
		\nor{f}{\oo} \le m \quad \text{and} \quad  \nor{f}{\oo} \pars{1+ \nor{Df}{\oo} + \nu \nor{D^2 f}{\oo}} \le K,
	\]
	and assume that $A > 1$,  $r_0\in (-1,0]$, $\eps_1,\eps_2: \Omega \to (0,1)$ and $-1 + \eps_2 \le r_0 \le 0$. Suppose that, for some $R \in (0,\oo]$, $w$ solves 
	\[
		\begin{dcases}
			\del_tw - \frac{\eps_2}{\eps_1^2} \nu m_+(D^2 w)
			+ \frac{1}{A} |Dw|^q - \pars{ \frac{\eps_2}{\eps_1}}^{q'} A \le \pars{ \frac{\eps_2}{\eps_1}}^{q'} f(\eps_1 x) \cdot \dot B(r_0 + \eps_2 t) & \text{and}\\
			\del_tw - \frac{\eps_2}{\eps_1^2} \nu m_-(D^2 w)
			+ A|Dw|^q + \pars{ \frac{\eps_2}{\eps_1}}^{q'}A \ge \pars{ \frac{\eps_2}{\eps_1}}^{q'} f(\eps_1 x) \cdot \dot B(r_0 + \eps_2 t) & \text{in } B_R \times [-1,0],
		\end{dcases}
	\]
	fix a  convex open set $\mcl C \subset B_R$, $x_0 \in \mcl C$, and $-1 \le t_1 < t_0 \le 0$. Then
	\[
		w_-(x_0,t_0) - \frac{\eps_2^{q'-1+\kappa}}{\eps_1^{q'}} A\mcl E
		\le w(x_0,t_0) \le w_+(x_0,t_0) + \ \frac{\eps_2^{q'-1+\kappa}}{\eps_1^{q'}} A\mcl E,
	\]
	where
	\[
		\begin{dcases}
			\del_tw_{-} - \frac{\eps_2}{\eps_1^2} \nu m_-(D^2 w_-) + 2A |Dw_-|^q = 0 & \text{and} \\
			\del_t w_{+} - \frac{\eps_2}{\eps_1^2} \nu m_+(D^2 w_+) + \frac{1}{2A} |Dw_+|^q = 0 & \text{in } \mcl C \times (t_1,t_0), \text{ and}\\
			w_- = w_+ = w & \text{on } \del^*( \mcl C \times [t_1,t_0]).
		\end{dcases}
	\]
	\end{enumerate}
\end{lemma}

\begin{proof}
	{\it Step 1.} For $(x,t) \in B_R \times [0,1]$, define $\tilde w(x,t) := w(x,-t)$, $\tilde w_\pm(x,t) := w_\pm(x,-t)$, and $\tilde B(t) := B(0) - B(-t)$. Then $\tilde B: [0,1] \times \Omega \to \RR^m$ is a Brownian motion, and $\tilde w$, $\tilde w_\pm$ solve
	\[
		\begin{dcases}
			-\del_t\tilde w - \frac{\eps_2}{\eps_1^2} \nu m_+(D^2 \tilde w)
			+ \frac{1}{A} |D\tilde w|^q - \pars{ \frac{\eps_2}{\eps_1}}^{q'} A \le  \pars{ \frac{\eps_2}{\eps_1}}^{q'} f(\eps_1 x) \cdot \dot {\tilde B}(-r_0 + \eps_2 t) & \text{and}\\
			-\del_t\tilde w - \frac{\eps_2}{\eps_1^2} \nu m_-(D^2 \tilde w)
			+ A|D\tilde w|^q +  \pars{ \frac{\eps_2}{\eps_1}}^{q'}A \ge  \pars{ \frac{\eps_2}{\eps_1}}^{q'} f(\eps_1 x) \cdot \dot {\tilde B}(-r_0 + \eps_2 t) & \text{in } B_R \times [0,1]
		\end{dcases}
	\]
	and
	\[
		\begin{dcases}
			- \del_t\tilde w_{-} - \frac{\eps_2}{\eps_1^2} \nu m_-(D^2 \tilde w_-) + 2A |D\tilde w_-|^q = 0 & \text{and} \\
			- \del_t\tilde w_{+} - \frac{\eps_2}{\eps_1^2} \nu m_+(D^2 \tilde w_+) + \frac{1}{2A} |D\tilde w_+|^q = 0 & \text{in } \mcl C \times [-t_0,-t_1), \text{ and}\\
			\tilde w_- = \tilde w_+ = \tilde w & \text{on } (\mcl C \times \{-t_1\}) \cup (\del \mcl C \times [-t_0,-t_1]).
		\end{dcases}
	\]
	
	{\it Step 2.} Let $W: [0,1] \times \mcl A \to \RR$ be a Brownian motion defined on a probability space $(\mcl A, \mcl F, \mbb P)$ independent of $(\Omega,\mbf F, \mbf P)$, fix $(\alpha,\beta) \in \mathscr C \times \mathscr C_{\eps_1^{-1} \sqrt{2 \eps_2 \nu}}$, assume that $X: [-t_0,-t_1] \times \mcl A$ is adapted with respect to $W$ and
	\[
		dX_r = \alpha_r dr + \sigma_r dW_r \quad \text{in } [-t_0,-t_1],
	\]
	and let $\tau \in [-t_0,-t_1]$ be a $W$-stopping time. 
	
	For $r_0 - \eps_2 t_0 \le r \le -r_0 + \eps_2 \tau$, we then set
	\[
		\left\{
		\begin{split}
		\tilde X_r &= \eps_1 X\pars{ \frac{r +  r_0}{\eps_2}},\\
		\tilde \alpha_r &= \frac{\eps_1}{\eps_2} \alpha\pars{ \frac{r + r_0}{\eps_2}},\\
		\tilde \sigma_r &= \frac{\eps_1}{\eps_2^{1/2}} \sigma \pars{ \frac{r + r_0}{\eps_2}}, \quad \text{and}\\
		\tilde W_r &= \eps_2^{1/2} \left[ W \pars{ \frac{r + r_0}{\eps_2}} - W(-t_0) \right],
		\end{split}
		\right.
	\]
	and we let $\tilde {\mathscr C}$ and $\tilde{\mathscr C}_M$ be defined just as $\mathscr C$ and $\mathscr C_M$, but with respect to the filtration of the Brownian motion $\tilde W$. Then $(\tilde \alpha, \tilde \sigma) \in \tilde{\mathscr C} \times \tilde{\mathscr C}_{\sqrt{2\nu}}$, $\tilde X$ is adapted with respect to $\tilde W$, $-r_0 + \eps_2 \tau$ is a $\tilde W$-stopping time, and
	\[
		d\tilde X_r = \tilde \alpha_r dr +\tilde \sigma_r d\tilde W_r \quad \text{for } - r_0 - \eps_2 t_0 \le r \le -r_0 + \eps_2 \tau.
	\]
	It now follows from Lemma \ref{L:intcontrol2} that, for some $\mcl E$ as in the statement of that lemma, and for all $0 < \delta \le 1$,
	\begin{equation}\label{intcontroluse2}
	\begin{split}
		&\abs{ \mbb E \left[ \pars{ \frac{\eps_2}{\eps_1}}^{q'} \int_{-t_0}^\tau f(\eps_1 X_r) \cdot \dot{\tilde B}(-r_0 +\eps_2 r)dr \right]}\\
		&= \frac{\eps_2^{q'-1}}{\eps_1^{q'}} \abs{ \mbb E \int_{-r_0 - \eps_2 t_0}^{-r_0 + \eps_2 \tau} f(\tilde X_r) \cdot \dot{\tilde B}(r)dr }\\
		&\le \frac{\eps_2^{q'-1}}{\eps_1^{q'}} \delta^{q'}  \mbb E \int_{-r_0 - \eps_2 t_0}^{-r_0 + \eps_2 \tau} \abs{ \tilde \alpha_r}^{q'}dr + \frac{\eps_2^{q'-1 + \kappa}}{\eps_1^{q'}} \frac{\mcl E}{\delta^q} (\tau + t_0)^\kappa \\
		&= \delta^{q'}  \mbb E  \int_{- t_0}^{\tau} \abs{ \alpha_r}^{q'}dr + \frac{\eps_2^{q'-1 + \kappa}}{\eps_1^{q'}} \frac{\mcl E}{\delta^q} (\tau + t_0)^\kappa.
	\end{split}
	\end{equation}
	
	{\it Step 3.} We now proceed with the proof of the lower bound. By Lemma \ref{L:HJBformula}(a), we have
	\begin{equation}\label{tildewinhomogge}
	\begin{split}
		\tilde w(x_0,-t_0) \ge \inf_{(\alpha, \sigma) \in \mathscr C \times \mathscr C_{\eps_1^{-1} \sqrt{2\eps_2 \nu}} } &\mbb E \left[ \tilde w(X_\tau,\tau) + c_q A^{-(q'-1)} \int_{-t_0}^\tau |\alpha_r|^{q'}dr - \pars{ \frac{\eps_2}{\eps_1}}^{q'} A(\tau + t_0) \right. \\
		&+ \left. \pars{ \frac{\eps_2}{\eps_1}}^{q'} \int_{-t_0}^\tau f(\eps_1 X_r) \cdot \dot{\tilde B}(-r_0 +\eps_2 r)dr \right],
	\end{split}
	\end{equation}
	where, as in that lemma, for fixed $(\alpha, \sigma) \in \mathscr C \times \mathscr C_{\eps_1^{-1} \sqrt{2\eps_2 \nu}}$, $X = X^{\alpha,\sigma}$ and $\tau = \tau^{\alpha,\sigma}$ satisfy
	\begin{equation}\label{stochdata}
		dX_r = \alpha_r dr + \sigma_r dW_r \quad \text{for } r \in [-t_0,-t_1], \quad X_{-t_0} = x_0, \quad \text{and} \quad
		\tau := \inf\left\{ t \in [-t_0, -t_1] : X_\tau \in \del \mcl C \right\}.
	\end{equation}
	We now set
	\[
		\delta := 1 \wedge c_q^{1/q'} (1 - 2^{-(q'-1)})^{1/q'} A^{-1/q},
	\]
	which implies, in particular, that $\delta^{q'} \le c_q (1 - 2^{-(q'-1)}) A^{-(q'-1)}$. Invoking \eqref{intcontroluse2}, we find that, for some constant $C_q > 0$,
	\begin{align*}
		& \mbb E \left[ \pars{ \frac{\eps_2}{\eps_1}}^{q'}  \int_{-t_0}^\tau f(\eps_1 X_r) \cdot \dot{\tilde B}(-r_0 +\eps_2 r)dr \right]\\
		&\ge -c_q (1- 2^{-(q'-1)} )A^{-(q'-1)}  \mbb E\int_{-t_0}^\tau |\alpha_r|^{q'} dr - C_q A \frac{\eps_2^{q'-1 + \kappa}}{\eps_1^{q'}} \mcl E.
	\end{align*}
	The inequality \eqref{tildewinhomogge} now becomes
	\begin{align*}
		\tilde w(x_0,-t_0) &\ge \inf_{(\alpha, \sigma) \in \mathscr C \times \mathscr C_{\eps_1^{-1} \sqrt{2\eps_2 \nu}} } \mbb E \left[ \tilde w(X_\tau,\tau) + c_q (2A)^{-(q'-1)} \int_{-t_0}^\tau |\alpha_r|^{q'}dr \right] \\
		& \qquad- \pars{ \frac{\eps_2}{\eps_1}}^{q'}  A \left[ 1 +  C_q \eps_2^{-(1-\kappa)}\mcl E \right]\\
		&\ge \tilde w_-(x_0,-t_0) -\frac{ \eps_2^{q'-1 + \kappa}}{\eps_1^{q'}} A (1 + C_q \mcl E).
	\end{align*}
	
	{\it Step 4.} We next obtain the upper bound. Lemma \ref{L:HJBformula}(b) gives
	\begin{equation}\label{tildewinhomogle}
	\begin{split}
		\tilde w(x_0,-t_0) \le \inf_{\alpha \in \mathscr C} \sup_{\beta \in \mathscr S_{\eps_1^{-1} \sqrt{2\eps_2 \nu}} } &\mbb E \left[ \tilde w(X_\tau,\tau) + c_q A^{q'-1} \int_{-t_0}^\tau |\alpha_r|^{q'}dr + \pars{ \frac{\eps_2}{\eps_1}}^{q'} A(\tau + t_0) \right. \\
		&+ \left. \pars{ \frac{\eps_2}{\eps_1}}^{q'} \int_{-t_0}^\tau f(\eps_1 X_r) \cdot \dot{\tilde B}(-r_0 +\eps_2 r)dr \right],
	\end{split}
	\end{equation}
	where, as in that lemma, for fixed $\alpha \in \mathscr C$ and $\beta \in \mathscr S_{\eps_1^{-1} \sqrt{2\eps_2 \nu}}$ with $\sigma = \beta(\alpha)$, $X = X^{\alpha,\sigma}$ and $\tau = \tau^{\alpha,\sigma}$ are as in \eqref{stochdata}. The inequality \eqref{intcontroluse2} then implies that, for all $\delta \in (0,1)$,
	\begin{align*}
		\tilde w(x_0,-t_0) \le \inf_{\alpha \in \mathscr C} \sup_{\beta \in \mathscr S_{\eps_1^{-1} \sqrt{2\eps_2 \nu}} } \mbb E \left[ \tilde w(X_\tau,\tau) + (c_q A^{q'-1} + \delta^{q'})  \int_{-t_0}^\tau |\alpha_r|^{q'}dr \right] + \pars{ \frac{\eps_2}{\eps_1}}^{q'} A  + \frac{\eps_2^{q'-1+\kappa}}{\eps_1^{q'}} \frac{\mcl E}{\delta^q}.
	\end{align*}
	We then set
	\[
		\delta = 1 \wedge (2^{q'-1} - 1)^{1/q'} c_q^{1/q'} A^{1/q},
	\]
	which, in particular, implies that $\delta^{q'} \le c_q (2^{q'-1} - 1) A^{q'-1}$, and so, for some $C_q' > 0$,
	\begin{align*}
		\tilde w(x_0,-t_0) &\le \inf_{\alpha \in \mathscr C} \sup_{\beta \in \mathscr S_{\eps_1^{-1} \sqrt{2\eps_2 \nu}} } \mbb E \left[ \tilde w(X_\tau,\tau) + c_q (2A)^{q'-1}\int_{-t_0}^\tau |\alpha_r|^{q'}dr \right] + \frac{\eps_2^{q'-1+\kappa}}{\eps_1^{q'}} A (1 + C_q' \mcl E)\\
		&= \tilde w_+(x_0,-t_0) + \frac{\eps_2^{q'-1+\kappa}}{\eps_1^{q'}} A (1 + C_q' \mcl E).
	\end{align*}
		The claimed upper bound for $w$ now follows from another time reversal. 
\end{proof}

We now introduce some smooth sub- and super-solutions of the homogenous second order equations that arise in the previous result, which will be used in Section \ref{S:secondorder}. The following lemma is proved in \cite{CS}, in particular, as Lemmas 4.2 and 4.6 and Corollary 4.3.

\begin{lemma}\label{L:CSbarriers}
	Let $q > 2$ and $A > 1$. Then there exist $C = C(q,A,d) > 0$ (which can be chosen arbitrarily large), $\nu_0 = \nu_0(q,A,d) > 0$ (which can be chosen arbitrarily small), and $\theta_0 = \theta_0(q,A,d) > 0$ such that the following hold:
	\begin{enumerate}[(a)]
		\item If $\eta > 0$,
		\[
			U(x,t) := C \frac{ (|x|^2 + \eta t)^{q'/2}}{t^{q'-1}} \quad \text{for } (x,t) \in \RR^d \times (0,\oo),
		\]
		and $0 < \nu < \eta \nu_0$, then
		\[
			\del_tU - \nu m_+(D^2 U) + \frac{1}{2A} |DU|^q \ge 0 \quad \text{in } \RR^d \times (0,\oo).
		\]
		
		\item Let $R > 0$, and assume that $b: \RR \to \RR$ is smooth and nonincreasing, $b(\tau) =1$ for $\tau < 3/4$, and $b(\tau) = 0$ for $\tau > 1$. If $0 < \theta < \theta_0R^{q'}$ and
		\[
			V(x,t) := 3\theta b \pars{ \frac{|x|}{R} + \frac{t}{4} } - \frac{C\nu \theta}{R^2} t \quad \text{for } (x,t) \in \RR^d \times (0,1),
		\]
		then
		\[
			\del_tV - \nu m_-(D^2 V) + 2A |DV|^q \le 0 \quad \text{in } \RR^d \times (0,1).
		\]
	\end{enumerate}
\end{lemma}

\subsection{Improvement of oscillation}

The main tool used in this paper is to establish an improvement of oscillation of solutions on all small scales. The next result explains how this leads to H\"older regularity estimates.

\begin{lemma}\label{L:ioo}
	Let $R, \tau, c > 0$, assume that $u: B_R \times [-\tau, 0]$ satisfies
	\[
		0 \le u \le c \quad \text{on } B_R \times [-\tau,0],
	\]
	fix $\alpha \in (0,1)$, $\beta > 0$, $0 < \mu < 1$, and $0 < a < R$ and $0 < b < \tau$. Assume that, whenever $(x_0,t_0) \in B_{R-a} \times [-\tau+b,0]$, the function
	\[
		v(x,t) := \frac{u(x_0 + ax, t_0 + bt)}{c} \quad \text{for } (x,t) \in B_1 \times [-1,0]
	\]
	satisfies
	\[
		\text{if} \quad 0 < r \le 1 \quad \text{and} \quad \osc_{B_r \times [-r^\beta,0]} v \le r^\alpha, \quad \text{then} \quad  \osc_{B_{\mu r} \times [-(\mu r)^\beta,0]} \le (\mu r)^\alpha.
	\]
	Then
	\[
		 \sup_{(x,t),(\tilde x, \tilde t) \in B_{R-a} \times [-\tau+b,0]}  \frac{|u(x,t) - u(\tilde x, \tilde t)|}{|x-\tilde x|^\alpha + |t - \tilde t|^{\alpha/\beta}} \le \frac{c}{\mu^\alpha} \pars{ \frac{1}{a^\alpha} \vee \frac{1}{b^{\alpha/\beta}}}.
	\]
\end{lemma}

\begin{proof}
	Choose $(x_0,t_0) \in B_{R - a} \times [-\tau + b,0]$ and define $v$ as in the statement of the lemma. Then $\osc_{B_1 \times [-1,0]} v \le 1$, and so an inductive argument implies that
	\[
		\osc_{B_{\mu^k} \times [-\mu^{k\beta},0]} v \le \mu^{k\alpha} \quad \text{for all } k = 0,1,2,\ldots
	\]
	Now choose $r \in (0,1]$ and let $k \in \NN$ be such that $\mu^{k+1} < r \le \mu^k$. Then
	\[
		\osc_{B_r \times [-r^\beta,0]} v \le \mu^{k\alpha} \le \frac{r^\alpha}{\mu^\alpha}.
	\]
	Fix $(y,s) \in B_1 \times [-1,0]$ and set $r := |y| \vee |s|^{1/\beta}$. We then have
	\[
		|v(0,0) - v(y,s)| \le \frac{r^\alpha}{\mu^\alpha} \le \frac{|y|^\alpha \vee |s|^{\alpha/\beta}}{\mu^\alpha}.
	\]
	Rescaling back to $u$, this means that, whenever  $(x,t),(\tilde x,\tilde t) \in B_{R-a} \times [-\tau+b,0]$ satisfy 
	\[
		|x - \tilde x| \le a \quad \text{and} \quad |t - \tilde t| \le b,
	\]
	we have
	\[
		|u(x,s) - u(\tilde x, \tilde t)| \le \frac{c}{\mu^\alpha} \pars{ \frac{1}{a^\alpha} \vee \frac{1}{b^{\alpha/\beta}}} \pars{ |x-\tilde x|^\alpha + |t - \tilde t|^{\alpha/\beta}}.
	\]
	The result now follows easily, because, for $|x - \tilde x| > a$,
	\[
		\frac{|u(x,t) - u(\tilde x, \tilde t)|}{|x-\tilde x|^\alpha + |t - \tilde t|^{\alpha/\beta}} \le \frac{c}{a^\alpha}
	\]
	and if $|t - \tilde t| > b$, then
	\[
		\frac{|u(x,t) - u(\tilde x, \tilde t)|}{|x-\tilde x|^\alpha + |t - \tilde t|^{\alpha/\beta}} \le \frac{c}{b^{\alpha/\beta}}.
	\]
\end{proof}

\section{First order equations}\label{S:firstorder}

In this section, we prove the regularity results for first order equations. We assume that
\begin{equation}\label{A:BM1}
	B: [-1,0] \times \Omega \to \RR^m \text{ is a standard Brownian motion on some probability space } (\Omega, \mbf F, \mbf P),
\end{equation}
and, for fixed
\begin{equation}\label{A:parameters}
	K > 0, \quad A > 1, \quad q > 1, \quad \text{and} \quad \mcl S: \Omega \to [0,\oo),
\end{equation}
we assume that
\begin{equation}\label{A:fauxHolderbound}
	f \in C^1(\RR^d \times \RR^m) \quad \text{and} \quad \nor{f}{\oo} + \nor{f}{\oo} \nor{Df}{\oo} \le K
\end{equation}
and
\begin{equation}\label{E:firstorder}
	\left\{
	\begin{split}
	&du + \left[\frac{1}{A} |Du|^q - A\right] dt \le \sum_{i=1}^m f^i(x) dB^i(t), \\
	&du + \left[ A |Du|^q + A \right] dt \ge \sum_{i=1}^m f^i(x) dB^i(t), \quad \text{and} \\
	&0 \le u \le \mcl S \quad \text{in } B_1 \times [-1,0].
	\end{split}
	\right.
\end{equation}

%Assume
%\begin{equation}\label{A:superlinearH}
%	\left\{
%	\begin{split}
%	&\text{for some $A > 1$ and $q > 1$ and all $(p,x,t) \in \RR^d \times \RR^d \times [0,\oo)$,}\\
%	&\frac{1}{A} |p|^q - A \le H(p,x,t) \le A \pars{ |p|^q + 1},
%	\end{split}
%	\right.
%\end{equation}

%Let $u$ be a solution of 
%\begin{equation}\label{E:firstorder}
%	du + H(Du,x,t)dt = \sum_{i=1}^m f^i(x) dB^i(t) \quad \text{in } B_1 \times [-1,0],
%\end{equation}
%such that
%\begin{equation}\label{A:ulinfty}
%	\nor{u}{\oo, B_1 \times [-1,0]} \le \frac{M + \mcl S}{2} \quad \text{for a constant $M \ge 1$ and a random variable $\mcl S: \Omega \to [0,\oo)$.}
%\end{equation}

\begin{theorem}\label{T:firstorder}
	Assume \eqref{A:BM1} - \eqref{E:firstorder}, and let $0 < \kappa < 1/2$ and $M \ge 1$. Then there exists $\alpha = \alpha(\kappa,A,q) \in (0,1)$, $c = c(\kappa,\alpha,q) > 0$, $\lambda_0 = \lambda_0(\kappa,A,K,M,q) > 0$ and, for all $p \ge 1$, $C = C(\kappa,A, K, M, p,q) > 0$ such that, for all $\lambda \ge \lambda_0$,
	\[
		\mbf P \pars{ \sup_{(x,s),(y,t) \in B_{1/2}\times [-1/2,0] } \frac{ |u(x,s) - u(y,t)|}{|x-y|^\alpha + |s-t|^{\alpha/(q-\alpha(q-1))}} > \lambda} \le \mbf P \pars{ (\mcl S - M)_+ > c\lambda^{1 -\alpha/q'} } + \frac{ C \nor{f}{\oo}^p}{ \lambda^{\kappa(q - \alpha(q-1)) p}}
	\]
\end{theorem}

\begin{proof}
	We first specify the parameters that determine the H\"older exponents, which depend only on $\kappa$, $A$, and $q$. Choose $\mu$ so that
	\begin{equation}\label{mu1}
		0 < \mu < \frac{1}{2} \quad \text{and} \quad 
		 \frac{1}{2} 12^{q'} c_q A^{q'-1} \mu^{q'} < 1,
	\end{equation}
	and then take $\theta$ sufficiently small that
	\begin{equation}\label{theta1}
		0 < \theta < \frac{1}{2}, \quad \frac{1}{2} 12^{q'} c_q A^{q'-1} \mu^{q'}  \le 1 - 4\theta,
		\quad \text{and} \quad
		2\theta \le c_q (2A)^{1-q'} \mu^{q'}.
	\end{equation}
	We now set
	\begin{equation}\label{alpha1}
		\alpha = \min \left(  \frac{ \log(1-\theta)}{\log \mu}, \frac{\kappa q}{\kappa q + 1 - \kappa} \right)
	\end{equation}
	and
	\begin{equation}\label{beta}
		\beta := q - \alpha(q-1).
	\end{equation}
	Note that $\beta - \alpha = q(1-\alpha) > 0$, and \eqref{alpha1} and \eqref{beta} together imply that $\beta\kappa - \alpha > 0$.
	
	We next identify a random scale $\rho$ at which the improvement of oscillation effect is seen. Let $\mcl D$ be the random variable as in Lemma \ref{L:Dbarrier}, set
	\[
		 \hat{\mcl S} := 1 \vee \mcl S,
	\]
and define
	\begin{equation}\label{rho1}
		\rho := \frac{1}{2 \hat{\mcl S}} \wedge \pars{ \frac{\theta}{A\mcl D}}^{\frac{1}{\kappa q}} .
	\end{equation}
	Note then that
	\[
		\rho \le 1, \quad \rho \hat{\mcl S} \le \frac{1}{2}, \quad \text{and} \quad \rho^{\kappa q} A\mcl D \le \theta.
	\]
	
	In what follows, for $(x_0,t_0) \in \RR^d \times \RR$, we define
	\[
		Q_r(x_0,t_0) := B_r(x_0) \times [t_0 - r^\beta, t_0] \quad \text{and} \quad Q_r := Q_r(0,0).
	\]
	
	{\it Step 1: The initial zoom-in.} Fix $(x_0,t_0) \in B_{1/2} \times [-1/2,0]$ and set
	\[
		v(x,t) := \frac{u(x_0 + \rho \hat{\mcl S} x, t_0 + \rho^q  \hat{\mcl S} t)}{ \hat{\mcl S}},
	\]
	which is well-defined for $(x,t) \in B_1 \times [-1,0]$ in view of \eqref{rho1}. Then $v$ satisfies
	\begin{equation}\label{E:vsystem}
		\left\{
		\begin{split}
		&\del_tv + A |Dv|^q + \rho^q A \ge \rho^q f(x_0 + \rho  \hat{\mcl S} x) \cdot \dot B(t_0 + \rho^q  \hat{\mcl S} t),\\
		&\del_tv + \frac{1}{A}|Dv|^q - \rho^q A \le \rho^q f(x_0 + \rho \hat{\mcl S} x) \cdot \dot B(t_0 + \rho^q  \hat{\mcl S} t), \quad \text{and}\\
		&0 \le v \le 1 \text{ in } B_1 \times [-1,0].
		\end{split}
		\right.
	\end{equation}	
	
	{\it Step 2: Induction step.} We next show that
	\begin{equation}\label{oscimprove1}
		\text{if} \quad 0 < r \le 1 \quad \text{and} \quad \osc_{Q_r} v \le r^\alpha, \quad \text{then} \quad \osc_{Q_{\mu r}} v \le (\mu r)^\alpha.
	\end{equation}
	Let $r \in (0,1]$ be such that $\osc_{Q_r} v \le r^\alpha$. We then set
	\[
		w(x,t) := \frac{ v(rx,r^\beta t) - \inf_{Q_r} v }{r^\alpha} \quad \text{for } (x,t) \in Q_1,
	\]
	which satisfies
	\[
	\left\{
	\begin{split}
		&\del_tw + \frac{1}{A} |Dw|^q - \pars{ \frac{\eps_2}{\eps_1}}^{q'} A \le \pars{ \frac{\eps_2}{\eps_1}}^{q'}  f(x_0 + \eps_1 x) \cdot \dot B(t_0 + \eps_2 t),\\
		&\del_tw + A|Dw|^q + \pars{ \frac{\eps_2}{\eps_1}}^{q'}  A \ge \pars{ \frac{\eps_2}{\eps_1}}^{q'}  f(x_0 + \eps_1 x) \cdot \dot B(t_0 + \eps_2 t), \text{ and} \\
		& 0 \le w \le 1 \text{ in } B_1 \times [-1,0],
	\end{split}
	\right.
	\]
	where $\eps_1 :=  \hat{\mcl S} \rho r$ and $\eps_2 := \hat{\mcl S} \rho^q r^\beta$. As a consequence of \eqref{rho1}, the random variables $\eps_1$ and $\eps_2$ take values in $(0,1/2]$, so that the hypotheses in part (b) of Lemma \ref{L:Dbarrier} are satisfied. We also compute, using \eqref{alpha1} and \eqref{beta},
	\[
		\frac{\eps_2^{q'-1+\kappa}}{\eps_1^{q'}} = \frac{ ( \hat{\mcl S} \rho^q r^\beta)^{q' - 1 + \kappa}}{( \hat{\mcl S}\rho r)^{q'}} = \frac{\rho^{\kappa q} r^{\beta \kappa - \alpha}}{ \hat{\mcl S}} \le \rho^{\kappa q}.
	\]

	To prove \eqref{oscimprove1}, we show that either
	\begin{equation}\label{upperosc}
		w(x,t) \le 1 - \theta \quad \text{for all } (x,t) \in B_\mu \times [-\mu^\beta,0]
	\end{equation}
	or
	\begin{equation}\label{lowerosc}
		w(x,t) \ge \theta \quad \text{for all } (x,t) \in B_\mu \times [-\mu^\beta,0].
	\end{equation}
	We consider the two following cases:
	
	{\it Case 1.} Assume first that 
	\begin{equation}\label{case1}
		\inf_{B_{2\mu}} w(\cdot,-1) \le 2 \theta.
	\end{equation}
	Fix $(x,t) \in B_\mu \times [-\mu^\beta,0]$. Then, by Lemma \ref{L:Dbarrier}, we have
	\[
		w(x,t) \le w_+(x,t) + \frac{\eps_2^{q' - 1 + \kappa}}{\eps_1^{q'}} A \mcl D \le w_+(x,t) + \rho^{\kappa q} A \mcl D \le w_+(x,t) + \theta,
	\]
	where
	\[
		w_+(x,t) = \inf_{(y,s) \in \del^*(B_{2\mu} \times [-1,t])} \left\{ w(y,s) + c_q (2A)^{q'-1}  \frac{|x - y|^{q'}}{(t-s)^{q'-1}} \right\}.
	\]
	We have
	\[
		t+1 \ge 1 - \mu^\beta \ge \frac{1}{2} \quad \text{and} \quad |x-y|^{q'} \le 3^{q'} \mu^{q'} \quad \text{for all } y \in B_{2\mu},
	\]
	and so, by \eqref{theta1},
	\begin{align*}
		w_+(x,t) &\le \inf_{y \in B_{2 \mu}} \left\{ w(y,-1) + c_q (2A)^{q'-1}  \frac{|x - y|^{q'}}{(t+1)^{q'-1}} \right\} \\
		&\le 6^{q'} c_q (2A)^{q'-1} \mu^{q'} + \inf_{y \in B_{2\mu}} w(y,-1)\\
		&\le 1 - 4\theta + 2\theta = 1 - 2 \theta.
	\end{align*}
	It follows that $w(x,t) \le 1 - 2\theta + \theta = 1 - \theta$, and so \eqref{upperosc} holds in this case.
	
	{\it Case 2.} Assume now that 
	\begin{equation}\label{case2}
		w(y,-1) \ge 2\theta \quad \text{for all } y \in B_{2 \mu}.
	\end{equation}
	Let $(x,t) \in B_{\mu} \times [-\mu^\beta,0]$. Then, similarly as in Step 1, Lemma \ref{L:Dbarrier} gives
	\begin{align*}
		w(x,t) &\ge \inf_{(y,s) \in \del^*(B_{2\mu} \times [-1,t])} \left\{ w(y,s) +  c_q (2A)^{1-q'} \frac{|x-y|^{q'}}{(t-s)^{q'-1}} \right\} - \theta.
	\end{align*}
	If $y \in B_{2 \mu}$ and $s = -1$, then \eqref{case2} implies that
	\[
		 w(y,s) + c_q (2A)^{1-q'} \frac{|y-x|^{q'}}{(t-s)^{q'-1}} - \theta \ge 2\theta - \theta = \theta,
	\]
	while, if $s \in [-1,t]$ and $y \in \del B_{2\mu}$, then $|y-x| \ge \mu$, and so, using \eqref{theta1} and the fact that $w \ge 0$,
	\[
		w(y,s)  + c_q (2A)^{1-q'} \frac{|y-x|^{q'}}{(t-s)^{q'-1}} - \theta
		\ge - \theta + c_q (2A)^{1-q'} \mu^{q'} \ge \theta.
	\]
	Either way, it is evident that \eqref{lowerosc} holds.
	
	Combining \eqref{upperosc} and \eqref{lowerosc}  with the definition of $\alpha$ in \eqref{alpha1}, we obtain
	\[
		\osc_{Q_\mu} w \le 1 - \theta \le \mu^\alpha,
	\]
	which, after rescaling back to $v$, yields
	\[
		\osc_{Q_{\mu r}} v \le (\mu r)^\alpha.
	\]
	
	{\it Step 3: the H\"older estimate.} We now invoke Lemma \ref{L:ioo} with the values
	\[
		a := \rho \hat{\mcl S}, \quad b := \rho^q \hat{\mcl S}, \quad \text{and} \quad c :=  \hat{\mcl S},
	\]
	and, using \eqref{mu1} and \eqref{rho1}, we get, for some constant $C_1 = C_1(\kappa, A, q) > 0$,
	\begin{align*}
		\sup_{(x,t),(\tilde x,\tilde t) \in B_{1/2} \times [-1/2,0]} \frac{|u(x,t) - u(\tilde x, \tilde t)|}{|x - \tilde x|^\alpha + |t - \tilde t|^{\alpha/\beta} }
		&\le \frac{c}{\mu^\alpha} \pars{ \frac{1}{a^\alpha} \vee \frac{1}{b^{\alpha/\beta}}}
		= \frac{1}{\mu^\alpha} \pars{ \frac{ \hat{\mcl S}^{1-\alpha}}{\rho^\alpha} \vee \frac{ \hat{\mcl S}^{1 - \alpha/\beta}}{\rho^{q\alpha/\beta}} }\\
		&\le \frac{1}{\mu^\alpha} \pars{ \frac{1}{2^{1-\alpha} \rho} \vee \frac{1}{2^{1-\alpha/\beta} \rho^{1+(q-1)\alpha/\beta} } }
		\le C_1 \rho^{-q/\beta}.
	\end{align*}
	In view of \eqref{theta1} and \eqref{rho1}, for some $C_2 = C_2(\kappa, A, q) > 0$,
	\[
		\rho^{-q/\beta} = (2  \hat{\mcl S})^{q/\beta} \vee \pars{ \frac{A \mcl D}{\theta}}^{\frac{1}{\kappa \beta}}
		\le C_2 \pars{  \hat{\mcl S}^{q/\beta} + \mcl D^{\frac{1}{\kappa \beta}} }.
	\]
	Since $M$ is chosen to be larger than $1$, we have $(\hat{\mcl S} - M)_+ = (\mcl S - M)_+$, and so, for some $C_3 = C_3(\kappa, A, q) > 0$,
	\[
		\sup_{(x,t),(\tilde x,\tilde t) \in B_{1/2} \times [-1/2,0]} \frac{|u(x,t) - u(\tilde x, \tilde t)|}{|x - \tilde x|^\alpha + |t - \tilde t|^{\alpha/\beta} }
		\le C_3\pars{ M^{q/\beta} + (\mcl S - M)_+^{q/\beta} + \mcl D^{\frac{1}{\kappa\beta}} }.
	\]
	Therefore, for any $\lambda > 0$,
	\begin{align*}
		\mbf P &\pars{ \sup_{(x,t),(\tilde x,\tilde t) \in B_{1/2} \times [-1/2,0]} \frac{|u(x,t) - u(\tilde x, \tilde t)|}{|x - \tilde x|^\alpha + |t - \tilde t|^{\alpha/\beta} } > \lambda}\\
		&\le \mbf P \pars{ (\mcl S - M)_+^{q/\beta} + \mcl D^{\frac{1}{\kappa\beta}} > \frac{\lambda - C_3 M^{q/\beta}}{C_3}}\\
		&\le \mbf P \pars{ (\mcl S - M)_+^{q/\beta} > \frac{\lambda - C_3 M^{q/\beta}}{2C_3}} + \mbf P \pars{ \mcl D^{\frac{1}{\kappa\beta}} > \frac{\lambda - C_3 M^{q/\beta}}{2C_3}}.
	\end{align*}
	Taking  $\lambda > 2C_3M^{q/\beta}$ yields
	\[
		\frac{\lambda - C_3 M^{q/\beta}}{2C_3} > \frac{\lambda}{4C_3},
	\]
	so that
	\[
		\mbf P \pars{ \sup_{(x,t),(\tilde x,\tilde t) \in B_{1/2} \times [-1/2,0]} \frac{|u(x,t) - u(\tilde x, \tilde t)|}{|x - \tilde x|^\alpha + |t - \tilde t|^{\alpha/\beta} } > \lambda} \le \mbf P \pars{ (\mcl S - M)_+^{q/\beta} > \frac{\lambda}{4C_3}} + \mbf P \pars{ \mcl D^{\frac{1}{\kappa\beta}} > \frac{\lambda}{4C_3}}.
	\]
	Finally, if $\lambda_0$ is as in Lemma \ref{L:Dbarrier}, then further taking $\lambda > 4C_3 \lambda_0^{1/(\kappa\beta)}$ yields the claim in view of the properties of $\mcl D$.
	\end{proof}

\section{Second order equations}\label{S:secondorder}
We now turn to the case of second order equations. We let $B$ be a Brownian motion as in \eqref{A:BM1}, and, for fixed
\begin{equation}\label{A:parameters2}
	\nu > 0, \quad K > 0, \quad A > 1, \quad q > 2, \quad \text{and} \quad \mcl S: \Omega \to [0,\oo),
\end{equation}
we assume that
\begin{equation}\label{A:f2}
	f \in C^2(\RR^d, \RR^m) \quad \text{and} \quad \nu + \nor{f}{\oo} + \nor{f}{\oo} \nor{Df}{\oo} + \nu \nor{f}{\oo} \nor{D^2 f}{\oo} \le K
\end{equation}
and
\begin{equation}\label{A:secondorder}
	\left\{
	\begin{split}
		&du + \left[ - \nu m_+(D^2u) + \frac{1}{A} |Du|^q - A\right] dt \le \sum_{i=1}^m f^i(x) \cdot dB^i(t), \\
		&du + \left[ - \nu m_-(D^2 u) + A |Du|^q + A \right]dt \ge \sum_{i=1}^m f^i(x) \cdot dB^i(t), \quad \text{and} \\
		&0 \le u \le \mcl S \quad \text{in } B_1 \times [-1,0].
	\end{split}
	\right.
\end{equation}

\begin{theorem}\label{T:secondorder}
	Assume \eqref{A:BM1} and \eqref{A:parameters2} - \eqref{A:secondorder}, and let $0 < \kappa < 1/2$ and $M \ge 1$. Then there exists $\alpha = \alpha(\kappa,A,q) \in (0,1)$, $c = c(\kappa,\alpha,q) > 0$, $\lambda_0 = \lambda_0(\kappa,A,K,M,q) > 0$, and, for all $p \ge 1$, $C = C(\kappa,A, K, M, p,q) > 0$ such that, for all $\lambda \ge \lambda_0$,
	\begin{align*}
		\mbf P &\pars{ \sup_{(x,s),(y,t) \in B_{1/2}\times [-1/2,0] } \frac{ |u(x,s) - u(y,t)|}{|x-y|^\alpha + |s-t|^{\alpha/(q-\alpha(q-1))}} > \lambda} \\
		&\le \mbf P \pars{ (\mcl S -M)_+ > c\lambda^{1 - \alpha/q'} } + C \frac{ \nor{f}{\oo}^p}{\lambda^{\kappa (q - \alpha(q-1))p}}.
	\end{align*}
\end{theorem}

\begin{proof}
We set up the various parameters similarly as in the proof of Theorem \ref{T:firstorder}, with a few changes to account for the second order terms.

We first choose $\mu$ such that 
\begin{equation}\label{mu2}
	0 < \mu < \frac{1}{4} \quad \text{and} \quad 
	 \frac{C}{2} 6^{q'} \mu^{q'}<1,
\end{equation}
where $C = C(q,A,d) > 4^{q'}$ is the constant from Lemma \ref{L:CSbarriers}, and we then take $\theta$ sufficiently small that
\begin{equation}\label{theta2}
	0 < \theta < \frac{1}{2}, \quad  \frac{C}{2} 6^{q'} \mu^{q'} \le 1 - 5\theta,
	\quad \text{and} \quad \theta < 4\mu^{q'} \theta_0,
\end{equation}
where $\theta_0 = \theta_0(q,A,d) > 0$ is as in Lemma \ref{L:CSbarriers}.

Set
\begin{equation}\label{alpha2}
	\alpha := \min \left\{ \frac{q-2}{q-1},  \frac{ \log(1-\theta)}{\log \mu}, \frac{\kappa q}{\kappa q + 1 - \kappa} \right\}
\end{equation}
and
\begin{equation}\label{beta2}
	\beta = q - \alpha(q-1).
\end{equation}
Observe that \eqref{alpha2} and \eqref{beta2} together imply that
\[
	1-\theta \le \mu^\alpha, \quad \beta - \alpha = q(1-\alpha), \quad \beta \kappa - \alpha \ge 0, \quad \text{and} \quad \beta \ge 2.
\]
As in the proof of Theorem \ref{T:firstorder}, we define, for $(x_0,t_0) \in \RR^d \times \RR$, 
	\[
		Q_r(x_0,t_0) := B_r(x_0) \times [t_0 - r^\beta, t_0] \quad \text{and} \quad Q_r := Q_r(0,0).
	\]

We now set
\[
	\hat{\mcl S} := \mcl S \vee 1,
\]
and, for $\mcl E$ the random variable from Lemma \ref{L:Ebarrier}, and $C$ and $\nu_0$ the values from Lemma \ref{L:CSbarriers}, the random variable $\rho$  is the largest value such that
\begin{equation}\label{rho2}
	\left\{
	\begin{split}
		(a) \quad &0 < \rho \le \frac{1}{2 \hat{\mcl S}},\\
		(b) \quad & \rho^{\kappa q} A \mcl E \le \theta, \\
		(c) \quad & 2^{q'-1} C K^{q'/2} \nu_0^{-q'/2} \rho^{q'(q-2)/2} \le \theta, \text{ and}\\
		(d) \quad & C \rho^{q-2} \le 4 \mu^2.
	\end{split}
	\right.
\end{equation}

{\it Step 1: The initial zoom-in.} Fix $(x_0,t_0) \in B_{1/2} \times [-1/2,0]$ and set
	\[
		v(x,t) := \frac{u(x_0 + \rho \hat{\mcl S} x, t_0 + \rho^q \hat{\mcl S} t)}{\hat{\mcl S}},
	\]
	which is well-defined for $(x,t) \in B_1 \times [-1,0]$ in view of \eqref{rho2}(a). Then $v$ satisfies
	\begin{equation}\label{E:vsystem2}
		\left\{
		\begin{split}
		&\del_tv - \frac{\nu \rho^{q-2}}{\hat{\mcl S}} m_-(D^2 v) + A |Dv|^q + \rho^q A \ge \rho^q f(x_0 + \rho \hat{\mcl S} x) \cdot \dot B(t_0 + \rho^q\hat{\mcl S}t),\\
		&\del_tv - \frac{\nu \rho^{q-2}}{\hat{\mcl S}} m_+(D^2 v) + \frac{1}{A}|Dv|^q - \rho^q A \le \rho^q f(x_0 + \rho\hat{\mcl S} x) \cdot \dot B(t_0 + \rho^q \hat{\mcl S} t), \quad \text{and}\\
		&0 \le v \le 1 \text{ in } B_1 \times [-1,0].
		\end{split}
		\right.
	\end{equation}

{\it Step 2: Induction step.} We next show that
\begin{equation}\label{oscimprove2}
	\text{if} \quad 0 < r \le 1 \quad \text{and} \quad \osc_{Q_r} v \le r^\alpha, \quad \text{then} \quad \osc_{Q_{\mu r}} v \le (\mu r)^\alpha.
\end{equation}

Let $r \in (0,1]$ be such that $\osc_{Q_r} v \le r^\alpha$. We then set
	\[
		w(x,t) := \frac{ v(rx,r^\beta t) - \inf_{Q_r} v }{r^\alpha} \quad \text{for } (x,t) \in B_1 \times [-1,0],
	\]
	which satisfies
	\[
	\left\{
	\begin{split}
		&\del_tw - \frac{\eps_2}{\eps_1^2} \nu m_+(D^2 w) + \frac{1}{A} |Dw|^q - \pars{ \frac{\eps_2}{\eps_1}}^{q'} A \le \pars{ \frac{\eps_2}{\eps_1}}^{q'}  f(x_0 + \eps_1 x) \cdot \dot B(t_0 + \eps_2 t),\\
		&\del_tw - \frac{\eps_2}{\eps_1^2} \nu m_-(D^2 w) + A|Dw|^q + \pars{ \frac{\eps_2}{\eps_1}}^{q'}  A \ge \pars{ \frac{\eps_2}{\eps_1}}^{q'}  f(x_0 + \eps_1 x) \cdot \dot B(t_0 + \eps_2 t), \text{ and} \\
		& 0 \le w \le 1 \text{ in } B_1 \times [-1,0],
	\end{split}
	\right.
	\]
	where $\eps_1 := \hat{\mcl S}\rho r$ and $\eps_2 := \hat{\mcl S}\rho^q r^\beta$. It is a consequence of \eqref{rho2}(a) that  $\eps_1,\eps_2 \in (0,1/2]$, and, moreover, just as in the proof of Theorem \ref{T:firstorder}, using the fact that $\beta \kappa \ge \alpha$,
	\[
		\frac{\eps_2^{q' -1 + \kappa}}{\eps_1^{q'}} \le \rho^{\kappa q}.
	\]
	
	To prove \eqref{oscimprove2}, we show that either
	\begin{equation}\label{upperosc2}
		w(x,t) \le 1 - \theta \quad \text{for all } (x,t) \in B_\mu \times [-\mu^\beta,0]
	\end{equation}
	or
	\begin{equation}\label{lowerosc2}
		w(x,t) \ge \theta \quad \text{for all } (x,t) \in B_\mu \times [-\mu^\beta,0].
	\end{equation}
	We consider the two following cases:
	
	{\it Case 1.} Assume first that
	\[
		\inf_{y \in B_{2\mu}} w(y,-1) \le 2 \theta.
	\]
	Let $(\hat x,\hat t) \in B_\mu \times [-\mu^\beta,0]$. Then \eqref{rho2}(b) and the upper bound from Lemma \ref{L:Ebarrier} imply that
	\[
		w(\hat x,\hat t) \le w_+(\hat x,\hat t) + \frac{\eps_2^{q' - 1 + \kappa}}{\eps_1^{q'}}A\mcl E
		\le w_+(\hat x, \hat t) + \rho^{\kappa q} A \mcl E \le w_+(\hat x, \hat t) + \theta,
	\]
	where
	\begin{equation}\label{E:w+}
		\begin{dcases}
		\del_t w_{+} - \frac{\eps_2}{\eps_1^2} \nu m_+(D^2 w_+) + \frac{1}{2A}|Dw_+|^q = 0 & \text{in } B_{2\mu} \times (-1,0] \text{ and}\\
		w_+ = w & \text{on } \del^*(B_{2\mu} \times [-1,0]).
		\end{dcases}
	\end{equation}
	  Note that, by the maximum principle, we have $0\leq w_+\leq 1$.
	Let  $C\geq 4^{q'}$ and $\nu_0$ be as in Lemma \ref{L:CSbarriers}, and, for $y \in B_{2\mu}$ and $(x,t) \in B_{2\mu} \times [-1,0]$, set
	\[
		w_y(x,t) := w(y,-1) + \frac{C}{(t+1)^{q'-1}} \pars{ |x-y|^2 + \frac{K \rho^{q-2}}{\nu_0} (t+1)}^{q'/2}.
	\]
	We compute
	\[
		\frac{\eps_2}{\eps_1^2} \nu = \frac{\rho^{q-2} r^{\beta - 2} \nu}{\hat{\mcl S}} \le \nu \rho^{q-2} \le \pars{ \frac{K \rho^{q-2}}{\nu_0} } \nu_0,
	\]
	and therefore, by Lemma \ref{L:CSbarriers}(a), $w_y$ is a super-solution of \eqref{E:w+}.  In addition,
	\[
		w_y(x,-1) =
		\begin{dcases}
			+\oo & \text{if } x \ne y, \\
			w(y,-1) & \text{if } x = y,
		\end{dcases}
	\]
	and, for any $(x,t)\in \partial B_1\times [-1,0]$, 
	\[
		w_y(x,t) \geq C((1-2\mu)^2)^{\frac{q'}{2}} \geq C4^{-q'} \geq 1 \geq w_+(x,t), 
	\]
	in view of the choice of $C\geq 4^{q'}$ and of $\mu<1/4$. So $w_y\geq w_+$ in $B_1\times [-1,0]$ by the comparison principle.  Because $\hat t \in [-\mu^\beta ,0]$, it follows that $1 + \hat t > 1 - \mu^\beta > \frac{1}{2}$, and so,  by \eqref{theta2} and \eqref{rho2}(c),
	\begin{align*}
		w_+(\hat x, \hat t) &\le \inf_{y \in B_{2\mu}} \left\{ w(y,-1) + \frac{C}{(\hat t+1)^{q'-1}} \pars{ |\hat x-y|^2 + \frac{K \rho^{q-2}}{\nu_0} (\hat t+1)}^{q'/2}\right\}\\
		&\le \frac{1}{2} 6^{q'} C \mu^{q'} + 2^{q'-1} CK^{q'/2}\nu_0^{-q'/2} \rho^{q'(q-2)/2} + \inf_{y \in B_{2\mu}} w(y,-1)\\
		&\le 1 - 5\theta + \theta + 2\theta = 1 - 2\theta.
	\end{align*}
	We conclude that $w(\hat x, \hat t) \le 1 - 2\theta + \theta = 1-\theta$, so that \eqref{upperosc2} holds in this case.
	
	{\it Case 2.} We now assume that
	\[
		\inf_{y \in B_{2\mu}} w(y,-1) > 2\theta.
	\]
	Fix $(\hat x, \hat t) \in B_{\mu} \times [-\mu^\beta,0]$. As in Step 1, Lemma \ref{L:Ebarrier} gives
	\[
		w(\hat x, \hat t) \ge w_-(\hat x, \hat t) - \theta,
	\]
	where
	\begin{equation}\label{E:w-}
		\begin{dcases}
		\del_t w_{-} - \frac{\eps_2}{\eps_1^2} \nu m_-(D^2 w_-) + 2A|Dw_-|^q = 0 & \text{in } B_{2\mu} \times (-1,0] \text{ and}\\
		w_- = w & \text{on } \del^*(B_{2\mu} \times [-1,0]).
		\end{dcases}
	\end{equation}
	For $(x,t) \in B_{2\mu} \times [-1,0]$ and for $b$ and $C$ as in Lemma \ref{L:CSbarriers}(b), define
	\[
		V(x,t) = 3\theta b \pars{ \frac{|x|}{2\mu} + \frac{t+1}{4}} - \frac{C\rho^{q-2} \theta}{4\mu^2 }(t+1).
	\]
	Then, by \eqref{theta2} and Lemma \ref{L:CSbarriers}(b), $V$ is a sub-solution of \eqref{E:w-}. In addition,
	\[
		V \le 0 \quad \text{on } \del B_{2\mu} \times [-1,0] \quad \text{and} \quad V \le 2\theta \quad \text{on } B_{2\mu} \times \{-1\},
	\]
	and so $V \le w_-$ on $\del^*(B_{2\mu} \times [-1,0])$. The comparison principle now implies that $V \le w_-$ in all of $B_{2\mu} \times [-1,0]$, and, in particular, using \eqref{rho2}(d) and the fact that $b(3/4) = 1$ and $b$ is nonincreasing,
	\begin{align*}
		w_-(\hat x, \hat t) \ge V(\hat x, \hat t) = 3\theta b \pars{ \frac{|\hat x|}{2\mu} + \frac{\hat t+1}{4}} - \frac{C\rho^{q-2} \theta}{4\mu^2 }(\hat t+1)
		\ge 3\theta b \pars{ \frac{3}{4}} - \theta = 2\theta.
	\end{align*}
	Thus, in this case, \eqref{lowerosc2} holds.
	
	Whether \eqref{upperosc2} or \eqref{lowerosc2} is satisfied, we have
	\[
		\osc_{Q_{\mu r}} v = r^\alpha \osc_{Q_\mu} w \le (1-\theta)r^\alpha \le (\mu r)^\alpha,
	\]
	and so \eqref{oscimprove2} is established.
	
	{\it Step 3: the H\"older estimate.} As in the proof of Theorem \ref{T:firstorder}, we use Lemma \ref{L:ioo} and \eqref{rho2}(a) to conclude that, for some $C_1 = C_1(\kappa,A,q) > 0$,
	\begin{align*}
		\sup_{(x,t),(\tilde x,\tilde t) \in B_{1/2} \times [-1/2,0]} \frac{|u(x,t) - u(\tilde x, \tilde t)|}{|x - \tilde x|^\alpha + |t - \tilde t|^{\alpha/\beta} } \le C_1 \rho^{-q/\beta}.
	\end{align*}
	All of the parts of \eqref{rho2} imply that, for some $C_2 = C_2(\kappa, A,q) > 0$ and $C_3 = C_3(\kappa,A,K,q) > 0$,
	\[
		\rho^{-q/\beta} \le \tilde C_2(\hat{\mcl S}^{q/\beta} + \mcl E^{1/(\kappa\beta)}) + C_3,
	\]
	and the rest of the proof follows as in the proof of Theorem \ref{T:firstorder} and the properties of $\mcl E$ outlined in Lemma \ref{L:Ebarrier}.
\end{proof}

\section{Applications}\label{S:apps}
In this section, we show how Theorems \ref{T:firstorder} and \ref{T:secondorder} can be used to prove a regularizing effect for certain initial value problems. Moreover, the regularity estimates are independent of a certain large-range, long-time scaling, which is useful in the theory of homogenization.

We fix a finite time horizon $T > 0$ and an initial condition
\begin{equation}\label{A:u0}
	u_0 \in BUC(\RR^d).
\end{equation}
The uniform continuity of $u_0$ ensures the well-posedness of the equations below, but we note that the regularizing effects we prove depend only on $\nor{u_0}{\oo}$.

Throughout,
\begin{equation}\label{A:ivpBM}
	B: [0,T] \times \Omega \to \RR^m \quad \text{is a Brownian motion over a probability space } (\Omega,\mbf F, \mbf P).
\end{equation}

We first consider equations of first order, and we assume that, for some $A > 1$ and $q > 1$,
\begin{equation}\label{A:firstorderH}
	\left\{
	\begin{split}
	&H \in C(\RR^d \times \RR^d \times [0,\oo)) \text{ satisfies}\\
	&\frac{1}{A} |p|^q - A  \le H(p,x,t) \le A|p|^q + A \quad \text{for all } (p,x,t) \in \RR^d \times \RR^d \times [0,T],
	\end{split}
	\right.
\end{equation}
and
\begin{equation}\label{A:firstorderf}
	f \in C^1_b(\RR^d,\RR^m).
\end{equation}
For $0 < \eps < 1$, we consider solutions of the scaled, forced equation
\begin{equation}\label{E:scaledfirstorder}
	du^\eps + H\pars{ Du^\eps, \frac{x}{\eps}, \frac{t}{\eps} }dt = \eps^{1/2} \sum_{i=1}^m f^i\pars{ \frac{x}{\eps}} \cdot dB^i(t) \quad \text{in } \RR^d \times (0,T] \quad \text{and} \quad u^\eps(\cdot,0) = u_0 \quad \text{on } \RR^d,
\end{equation}
and we prove the following result:

\begin{theorem}\label{T:firstorderscaling}
	Assume \eqref{A:u0} - \eqref{A:firstorderf}, and, for  $0 < \eps \leq 1$, let $u^\eps$ be the solution of \eqref{E:scaledfirstorder}. Fix  $p\geq 1$, $\tau > 0$ and $R > 0$. Then there exist $C = C(R,\tau,T,A, \nor{f}{C^1},\nor{u}{\oo}, p, q) > 0$, $\alpha = \alpha(A,q) > 0$, and $\sigma = \sigma(A,q) > 0$ such that, for all $\lambda > 0$,
	\[
		\mbf P \pars{ \sup_{(x,t), (\tilde x,\tilde t) \in B_R \times [\tau,T]} \frac{|u^\eps(x,t) - u^\eps(\tilde x, \tilde t)|}{|x - \tilde x|^\alpha + |t - \tilde t|^{\alpha/(q - \alpha(q-1))} } > C + \lambda} \le \frac{C\eps^{p/2}}{\lambda^{\sigma p}}.
	\]
\end{theorem}

\begin{proof}
	We first note that we can assume, without loss of generality, that $\tau > 1/2$. Indeed, otherwise, we consider the function
	\[
		\tilde u^\eps(x,t) := \frac{1}{2\tau} u^\eps(2\tau x , 2\tau t) \quad \text{for } (x,t) \in \RR^d \times \left[0, \frac{T}{2\tau} \right],
	\]
	which solves
	\[
		d \tilde u^\eps + \tilde H\pars{ D \tilde u^\eps, \frac{x}{\eps}, \frac{t}{\eps}}dt = \eps^{1/2} \sum_{i=1}^m \tilde f^i\pars{ \frac{x}{\eps}} \cdot d\tilde B^i(t) \quad \text{in } \RR^d \times \left(0, \frac{T}{2\tau} \right) \quad \text{and} \quad \tilde u^\eps(\cdot,0) = \tilde u_0 \quad \text{on } \RR^d,
	\]
	where, for $(p,x,t) \in \RR^d \times \RR^d \times \left[0, \frac{T}{2\tau} \right]$,
	\[
		\tilde H(p,x,t) := H(p,2\tau x, 2 \tau t), \quad \tilde u_0(x) = \frac{1}{2\tau} u_0(2\tau x), \quad \tilde f(x) = \frac{1}{\sqrt{2\tau}} f(2\tau x), \quad \text{and} \quad \tilde B(t) = \frac{1}{\sqrt{2\tau}} B(2\tau t).
	\]
	Then $\tilde H$ satisfies \eqref{A:firstorderH} with $A$ and $q$ unchanged, and $\tilde B$ is a Brownian motion on $[0,2\tau T]$. As a consequence, $\alpha = \alpha(A,q) > 0$ remains unchanged, and the $\tau$-dependence can be absorbed into $R$, $T$, $\nor{f}{C^1}$, and $\nor{u_0}{\oo}$.
	
	Crucially, if $f^\eps(x) := \eps^{1/2} f(x/\eps)$, then
	\[
		\nor{f^\eps}{\oo} = \eps^{1/2} \nor{f}{\oo} \quad \text{and} \quad \nor{f^\eps}{\oo} \nor{Df^\eps}{\oo} = \nor{f}{\oo} \nor{Df}{\oo}.
	\]
	As a consequence, we may choose a fixed constant $K > 0$ such that the conclusions of Lemma \ref{L:Dbarrier} and Theorem \ref{T:firstorder} hold with the function $f^\eps$, for all $\eps \in( 0,1]$.
	
	In what follows, we fix $0 < \kappa < \frac{1}{2}$.
	
	{\it Step 1: $u$ is bounded.} We first use Lemma \ref{L:Dbarrier} to describe the $L^\oo$-bound for $u$ on $\RR^d \times [0,T]$. In view of \eqref{A:firstorderH}, Lemma \ref{L:Dbarrier} with $\eps_1 = \eps_2 = 1$, $R = +\oo$, and $\mcl C = \RR^d$ gives
	\[
		u^\eps(x,t) \le u_+(x,t) + A \mcl D_1 \quad \text{on } \RR^d \times [0,1],
	\]
	where, for some $\lambda_1 = \lambda_1(\kappa, \nor{f}{C^1}, q) > 0$ and, given $p \ge 1$, some $C = C(\kappa, \nor{f}{C^1}, p,q) > 0$,
	\[
		\mbf P \pars{ D_1 > \lambda} \le \frac{C \eps^{p/2}}{\lambda^p} \quad \text{for all } \lambda \ge \lambda_1
	\]
	and
	\[
		\del_tu_{+} + \frac{1}{2A} |Du_+|^{q} = 0 \quad \text{on } \RR^d \times[0,1], \quad \text{and } u_+(\cdot,0) = u_0 \quad \text{on } \RR^d.
	\]
	The comparison principle yields $u_+(x,t) \le \nor{u_0}{\oo}$. It follows that
	\[
		u^\eps(x,t) \le \nor{u_0}{\oo} + C(1 + \mcl D_1) \quad \text{on } \RR^d \times [0,1]
	\]
	Set $N := \lceil T \rceil$. An inductive argument then gives random variables $\mcl D_2, \mcl D_3, \ldots, \mcl D_N: \Omega \to \RR_+$ and $\lambda_2,\lambda_2, \ldots, \lambda_N$ depending on $\kappa$, $\nor{f}{C^1}$, and $q$ such that
	\[
		u^\eps(x,t) \le \nor{u_0}{\oo} + A\sum_{k=1}^N \mcl D_n \quad \text{on } \RR^d \times [0,T]
	\]
	and, for all $k = 1,2,\ldots, N$, $p \ge 1$, and some $C = C(\kappa, \nor{f}{C^1}, p,q) > 0$,
	\[
		\mbf P\pars{ D_k > \lambda} \le \frac{C\eps^{p/2}}{\lambda^p} \quad \text{for all }\lambda \ge \lambda_k.
	\]
	A similar argument, using the lower bound of Lemma \ref{L:Dbarrier}, gives
	\[
		u^\eps(x,t) \ge -\nor{u_0}{\oo} - A\sum_{k=1}^N \mcl D_n \quad \text{on } \RR^d \times [0,T].
	\]
	Adding a random constant to $u^\eps$, which does not affect the equation solved by $u^\eps$, we may then write
	\[
		0 \le u^\eps \le \mcl S \quad \text{on } \RR^d \times [0,T],
	\]
	where
	\[
		\mcl S := 2 \nor{u^\eps}{\oo} + 2A \sum_{k=1}^N \mcl D_k.
	\]
	Setting $M := 1 \vee (2\nor{u_0}{\oo})$, we then have, for all $p \ge 1$, $\lambda \ge \lambda_1 \vee \lambda_2 \vee \cdots \vee \lambda_N$, and some constant $C = C(\kappa,\nor{f}{C^1}, A, p,q,T) > 0$,
	\begin{equation}\label{Stails}
		\mbf P \pars{ (\mcl S-M)_+ > \lambda} \le \mbf P \pars{ \sum_{k=1}^N \mcl D_k > \frac{\lambda}{2A}}
		\le \frac{C \eps^{p/2}}{\lambda^p}.
	\end{equation}
	
	{\it Step 2: the H\"older estimate.} Because $\tau > 1/2$, we can cover $B_R \times [\tau,T]$ with cylinders on which, by Theorem \ref{T:firstorder}, $u$ is H\"older continuous. More precisely, there exists $\alpha$, $\lambda_0$ and $C$ as in the statement of the current theorem, and $c = c(\kappa, \alpha,q) > 0$, such that, for all $p \ge 1$ and $\lambda \ge \lambda_0$,
	\[
		\mbf P \pars{ \sup_{(x,t), (\tilde x, \tilde t) \in B_{R} \times [\tau,T]} \frac{ |u^\eps(x,t) - u^\eps(\tilde x, \tilde t)|}{|x - \tilde x|^\alpha + |t - \tilde t|^{\alpha/(q - \alpha(q-1))}} > \lambda} \le \mbf P((\mcl S-M)_+ > c\lambda^{1-\alpha/q'}) + \frac{C\eps^{p/2}}{\lambda^{\kappa (q - \alpha(q-1))p}}.
	\]
	Making $\lambda_0$ larger if necessary, depending on $\kappa$, $\nor{f}{C^1}$, and $q$, we invoke \eqref{Stails} and obtain the result with
	\[
		\sigma = \pars{ 1 - \frac{\alpha}{q'}} \wedge \pars{ \kappa (q - \alpha(q-1))} = (q - \alpha(q-1)) \pars{ \frac{1}{q} \wedge \kappa}.
	\]
\end{proof}

The next result is for the second-order case. Assume that, for some $A > 1$, $\nu > 0$, and $q > 2$,
\begin{equation}\label{A:secondorderF}
	\left\{
	\begin{split}
	&F \in C(\mbb S^d \times \RR^d \times \RR^d \times [0,\oo)) \text{ satisfies}\\
	&-\nu m_+(X) + \frac{1}{A} |p|^q - A  \le F(X,p,x,t) \le -\nu m_-(X) + A|p|^q + A\\
	&\quad \text{for all } (X,p,x,t) \in \mbb S^d \times \RR^d \times \RR^d \times [0,T],
	\end{split}
	\right.
\end{equation}
and
\begin{equation}\label{A:secondorderf}
	f \in C^2_b(\RR^d,\RR^m).
\end{equation}
For $0 < \eps < 1$, the scaled equation we consider is
\begin{equation}\label{E:scaledsecondorder}
	du^\eps + F\pars{\eps D^2u^\eps, Du^\eps, \frac{x}{\eps}, \frac{t}{\eps} }dt = \eps^{1/2} \sum_{i=1}^m f^i\pars{ \frac{x}{\eps}} \cdot dB^i(t) \quad \text{in } \RR^d \times (0,T] \quad \text{and} \quad u^\eps(\cdot,0) = u_0 \quad \text{on } \RR^d,
\end{equation}
and we prove the following result:

\begin{theorem}\label{T:secondorderscaling}
	Assume \eqref{A:u0}, \eqref{A:ivpBM}, \eqref{A:secondorderF}, and \eqref{A:secondorderf}, and, for  $0 < \eps \leq 1$, let $u^\eps$ be the solution of \eqref{E:scaledsecondorder}. Fix  $p\geq 1$, $\tau > 0$ and $R > 0$. Then there exists a constant $C = C(R,\tau,T,A, \nor{f}{C^2},\nor{u_0}{\oo}, p, q) > 0$, $\alpha = \alpha(A,q) > 0$, and $\sigma = \sigma(A,q) > 0$ such that
	\[
		\mbf P \pars{ \sup_{(x,t), (\tilde x,\tilde t) \in B_R \times [\tau,T]} \frac{|u^\eps(x,t) - u^\eps(\tilde x, \tilde t)|}{|x - \tilde x|^\alpha + |t - \tilde t|^{\alpha/(q - \alpha(q-1))} } > C + \lambda} \le \frac{C\eps^{p/2}}{\lambda^{ \sigma p}}.
	\]
\end{theorem}

\begin{proof}
	Arguing as in the proof of Theorem \ref{T:firstorderscaling}, we may assume without loss of generality that $\tau > 1/2$. Notice also that 
	\[
		F^\eps(X,p,x,t) := F \pars{ \eps X, p, \frac{x}{\eps}, \frac{t}{\eps}} \quad \text{for } (X,p,x,t) \in \mbb S^d \times \RR^d \times \RR^d \times [0,T]
	\]
	satisfies \eqref{A:firstorderH} with $\eps\nu$ replacing $\nu$, and, therefore, if we define $f^\eps(x) := \eps^{1/2} f(x/\eps)$, we have $\nor{f^\eps}{\oo} = \eps^{1/2} \nor{f}{\oo}$ and
	\[
		\eps \nu + \nor{f^\eps}{\oo} \nor{Df^\eps}{\oo} + \eps \nu \nor{f^\eps}{\oo} \nor{D^2 f^\eps}{\oo}  \le \nu + \nor{f}{\oo} \nor{Df}{\oo} + \nor{f}{\oo} \nor{D^2 f}{\oo}.
	\]
	As a consequence, we may choose a constant $K > 0$ independently of $\eps > 0$ for which the conclusions of Lemma \ref{L:Ebarrier} and Theorem \ref{T:secondorder} hold with the function $f^\eps$. The rest of the proof then follows exactly as in the proof of Theorem \ref{T:firstorderscaling}.
\end{proof}

\appendix

\section{Controlling stochastic integrals}\label{S:intcontrol}

Throughout the paper, we use the following results that give uniform control over certain stochastic integrals. Assume below that
\begin{equation}\label{appBM}
	B: [-1,0] \times \Omega \to \RR^m \quad \text{is a standard Brownian motion over the probability space } (\Omega, \mbf F, \mbf P).
\end{equation}

\begin{lemma}\label{L:intcontrol1}
Let $m > 0$, $K > 0$, $q > 1$, and $\kappa \in (0,1/2)$. Then there exists a random variable $\mcl D : \Omega \to \RR_+$ and $\lambda_0 =  \lambda_0(\kappa, K,q) > 0$ such that
	\begin{enumerate}[(a)]
	\item for any $p \ge 1$ and some constant $ C = C(\kappa,K,p,q) > 0$,
	\[
		\mbf P( \mcl D > \lambda) \le \frac{Cm^p}{\lambda^p} \quad \text{for all } \lambda \ge \lambda_0,
	\]
	and
	\item for all $\gamma \in W^{1,\oo}([-1,0],\RR^d)$, $\delta \in (0,1]$, $-1 \le s \le t \le 0$, and $f$ satisfying
	\[
		\nor{f}{\oo} \le m \quad \text{and} \quad  \nor{f}{\oo}(1+ \nor{Df}{\oo} )\le K,
	\]
	we have
	\[
		\abs{ \int_s^t f(\gamma_r) \cdot dB_r } \le \delta^{q'}\int_{s}^{t} |\dot \gamma_r|^{q'}dr + \frac{\mcl D}{\delta^{q} }(t-s)^\kappa.
	\]
	\end{enumerate}
\end{lemma}

Assume now that
\[
	W: [-1,0] \times \mcl A \to \RR \quad \text{is a Brownian motion defined over a probability space } (\mcl A, \mcl F, \mbb P).
\]
The probability space $\mcl A$ is independent of $\Omega$. Below, we prove a statement that is true for $\mbf P$-almost every sample path $B$ of the Brownian motion from \eqref{appBM}, which involves taking the expectation with respect to the Brownian motion $W$. Effectively, $B$ and $W$ are independent Brownian motions, and $\mbb E$ can be interpreted as the expectation conditioned with respect to $B$.

\begin{lemma}\label{L:intcontrol2}
	Let $q > 1$, $0 < \kappa < \frac{1}{2}$, $m > 0$, and $K > 0$. Then there exist a random variable $\mcl E: \Omega \to \RR_+$ and $\lambda_0 := \lambda_0(\kappa,K,q) > 0$ such that
	\begin{enumerate}[(a)]
	\item for any $p \ge 1$ and some constant $C = C(\kappa,K,p,q) > 0$, 
	\[
		\mbf P( \mcl E > \lambda) \le \frac{C m^p}{\lambda^p} \quad \text{for all } \lambda \ge \lambda_0,
	\]
	and
	\item for all $0 < \delta \le 1$; processes $(\alpha, \sigma,X) : [-1,0] \times \mcl A \to \RR^d \times \RR^d \times \RR^d$ that are $W$-adapted such that
	\begin{equation}\label{processes}
		\alpha, \sigma \in L^\oo([-1,0] \times \mcl A) \quad \text{and} \quad dX_r = \alpha_r dr + \sigma_r dW_r \quad \text{for } r \in [-1,0];
	\end{equation}
	$W$-stopping times $-1 \le s \le t \le 0$; and $f \in C^1(\RR^d,\RR^m)$ satisfying
	\begin{equation}\label{fs}
		\nor{f}{\oo} \le m \quad \text{and} \quad \nor{f}{\oo} \pars{1+ \nor{Df}{\oo} + \nor{\sigma \sigma^t}{\oo} \nor{D^2 f}{\oo} } \le K;
	\end{equation}
	we have
	\[
		\abs{ \mbb E \left[ \int_{s}^{t} f(X_r) \cdot dB_r  \right]} \le \delta^{q'} \mbb E \int_{s}^{t} |\alpha_r|^{q'}dr + \frac{\mcl E}{\delta^{q} }(t-s)^\kappa.
	\]
	\end{enumerate}
\end{lemma}

We note that the integrals against $dB$ appearing in Lemmas \ref{L:intcontrol1} and \ref{L:intcontrol2} are interpreted as in Section \ref{S:prelim}, and, in particular, subsection \ref{SS:formulae}.

The proof of Lemma \ref{L:intcontrol1} can be found in \cite{Shomog}. The arguments for Lemma \ref{L:intcontrol2} are similar, but some further details are needed to account for the use of It\^o's formula and the interaction between $B$ and $W$.

We first give a parameter-dependent variant of Kolmogorov's continuity criterion. Its statement and proof are very similar to that in \cite{Shomog}.

\begin{lemma}\label{L:Kolmogorov}
	Define $\triangle := \left\{ (s,t) \in [-1,0], \; s \le t \right\}$ and fix a parameter set $\mcl M$. Let $(M_\mu)_{\mu \in \mcl M}: \Omega \to \RR_+$ and $(Z_\mu)_{\mu \in \mcl M}: \triangle \times \Omega \to \RR_+$ be such that
	\begin{equation}\label{subadditive}
		Z_\mu(s,u) \le Z_\mu(s,t) + Z_\mu(t,u) \quad \text{for all } \mu \in \mcl M \text{ and } -1 \le s \le t \le u \le 0,
	\end{equation}
	and, for some constants $a > 0$, $\beta \in (0,1)$, $p \ge 1$,
	\[
		\sup_{(s,t) \in \triangle} \mbf E \left[\sup_{\mu \in \mcl M}\pars{  \frac{ Z_\mu(s,t)}{(t-s)^{\beta + 1/p}} - M_\mu}_+^p \right] \le a.
	\]
	Then, for all $0 < \kappa < \beta$, there exist $C_1 = C_1(\kappa) > 0$ and $C_2 = C_2(p,\kappa,\beta) > 0$ such that, for all $\lambda \ge 1$,
	\[
		\mbf P\pars{ \sup_{\mu \in \mcl M} \sup_{(s,t) \in \triangle} \pars{ \frac{Z_\mu(s,t)}{(t-s)^\kappa}  - C_1 M_\mu} > \lambda } \le \frac{C_2 a}{\lambda^p}.
	\]
\end{lemma}

The next result gives an estimate for moments of sums of certain centered and independent random variables.

\begin{lemma}\label{L:sums}
	Let $(Y_k)_{k=1}^n: \Omega \to \RR$ be a sequence of centered and independent random variables such that, for all $p \ge 1$ and for some $\mu > 0$ and $C = C(p) > 0$,
	\[
		\mbf E |Y_1|^p \le C \mu^p.
	\]
	Then there exists a constant $\tilde C = \tilde C(p) > 0$ such that
	\[
		\mbf E \abs{ \sum_{k=1}^n Y_k}^p \le \tilde C n^{p/2} \mu^p.
	\]
\end{lemma}

\begin{proof}
	Let $(\eps_k)_{k=1}^n$ be a sequence of independent Rademacher random variables, that is,
	\[
		\mbf P(\eps _k = 1) = \mbf P(\eps_k = -1) = \frac{1}{2} \quad \text{for all } k = 1,2,\ldots,n,
	\]
	such that $(\eps_k)_{k=1}^n$ is independent of the sequence $(Y_k)_{k=1}^n$. It then follows (see Kahane \cite{K}) that
	\[
		\mbf E \abs{ \sum_{k=1}^n Y_k}^p \le 2^p \mbf E \abs{ \sum_{k=1}^n \eps_k Y_k}^p.
	\]
	Therefore, upon replacing $Y_k$ with $\eps_k Y_k$, we may assume without loss of generality that each $Y_k$ is symmetric, that is, $Y_k$ and $-Y_k$ are identically distributed.
	
	Observe next that if the result holds for some $p \ge 1$, then, for any $q <p$, by H\"older's inequality,
	\[
		\mbf E \abs{\sum_{k=1}^n Y_k}^q \le \pars{\mbf E \abs{\sum_{k=1}^n Y_k}^p}^{q/p} \le \pars{\tilde C n^{p/2} \mu^p}^{q/p} \le \tilde C^{q/p} n^{q/2} \mu^q.
	\]
	Therefore, it suffices to prove the result for $p = 2m$ with $m \in \NN$. 
	
	We compute
	\[
		\abs{ \sum_{k=1}^n Y_k}^{2m} = \sum Y_{k_1}^{j_1} Y_{k_2}^{j_2} \cdots Y_{k_\ell}^{j_\ell},
	\]
	where the sum is taken over $1 \le k_1 < k_2 < \cdots < k_\ell \le n$ and $j_1 + j_2 + \cdots + j_\ell = 2m$. In view of the symmetry and independence of the $Y_k$, all summands for which one or more of the $j_i$ values is odd have zero expectation. Thus,
	\[
		\mbf E \abs{ \sum_{k=1}^n Y_k}^{2m} = \sum \mbf EY_{k_1}^{2i_1} Y_{k_2}^{2i_2} \cdots Y_{k_\ell}^{2i_\ell},
	\]
	where the sum is taken over $1 \le k_1 < k_2 < \cdots < k_\ell \le n$ and $i_1 + i_2 + \cdots + i_\ell = m$. A combinatorial argument implies that the cardinality of such terms is equal to $\binom{m + n-1}{n-1}$, while H\"older's inequality gives
	\[
		\mbf EY_{k_1}^{2i_1} Y_{k_2}^{2i_2} \cdots Y_{k_\ell}^{2i_\ell} 
		\le \pars{ \mbf E Y_{k_1}^{2m}}^{i_1/m} \pars{ \mbf E Y_{k_2}^{2m}}^{i_2/m} \cdots \pars{ \mbf E Y_{k_\ell}^{2m}}^{i_\ell/m} \le C \mu^{2m},
	\]
	and, therefore,
	\[
		\mbf E \abs{ \sum_{k=1}^n Y_k}^{2m} \le C\binom{m + n-1}{n-1} \mu^{2m} \le C n^m \mu^{2m}.
	\]	
\end{proof}

Finally, we turn to the proof of Lemma \ref{L:intcontrol2}.
\begin{proof}[Proof of Lemma \ref{L:intcontrol2}]
	Let $\mathscr C_{m,K}$ be the space consisting of $(\alpha,\sigma,X,f)$ satisfying \eqref{processes} and \eqref{fs}, define the parameter set
	\[
		\mcl M := (0,1) \times \mathscr C_{m,K},
	\]
	and, for each $\mu = (\delta,\alpha,\sigma,X,f) \in \mcl M$ and $(s,t) \in \triangle$, the stochastic process
	\begin{equation}\label{subadditiveguy}
		Z_\mu(s,t) := \pars{ \abs{  \mbb E \left[ \delta^{q} \int_{s}^t f(X_r)\cdot dB_r \right] } -\delta^{q+q'}  \mbb E  \int_{s}^t |\alpha_r|^{q'}dr }_+,
 	\end{equation}
	which can easily be seen to satisfy \eqref{subadditive}.
	
	We first show that there exist constants  $M_1 = M_1(K,q) > 0$ and $M_2 = M_2(K,p,q) > 0$ such that
	\begin{equation}\label{pointwisemoment}
		\sup_{-1 \le s \le t \le 0} \mbf E \left[ \sup_{\mu \in \mcl M} \pars{  \frac{ Z_\mu(s,t)}{ (t - s)^{1/2} } - M_1   }^p_+\right] \le M_2 m^p.
	\end{equation}
	
	Fix $s,t \in [-1,0]$ with $s \le t$. We split into two cases, depending on the size of the interval $[s,t]$.
		
	{\it Case 1.} Assume first that
\begin{equation}\label{Deltacase1}
	t-s \le \frac{\nor{f}{\oo}^q}{\nor{Df}{\oo}^q} \wedge \frac{\nor{f}{\oo}}{ \nor{\sigma \sigma^t}{\oo,[-1,0]} \nor{D^2 f}{\oo} }.
\end{equation}

By Lemma \ref{L:intzeta2},
\begin{align*}
	\mbb E \left[ \int_{s}^{t} f(X_r) \cdot dB_r \right]
	&= \mbb E \left[ f(X_t)\cdot (B_t - B_s) \right]\\
	&- \mbb E \left[ \int_{s}^{t} \pars{ Df(\gamma_r) \cdot  \alpha_r + \frac{1}{2} \tr(\sigma_r \sigma_r^t D^2 f(X_r)) } \cdot (B_r - B_{s})dr \right].
\end{align*}
Setting
\[
	\Delta := \max_{r_1,r_2 \in [s,t]} \abs{ B_{r_1} - B_{r_2}}
\]
and invoking \eqref{Deltacase1} and the Young and H\"older inequalities then gives, for some constant $C = C(K, q) > 0$,
\begin{align*}
	&\abs{ \mbb E \left[  \int_s^t f(X_r) \cdot dB_r  \right] } \\
	&\le \nor{f}{\oo} \Delta  + \nor{Df}{\oo} \Delta  \mbb E  \int_s^t |\alpha_r|dr + \frac{1}{2}\nor{\sigma \sigma^t}{\oo} \nor{D^2 f}{\oo} \Delta (t-s)\\
	&\le \nor{f}{\oo} \pars{\frac{3}{2} \Delta + \Delta\pars{ \mbb E \int_{s}^{t} |\alpha_r|^{q'}dr}^{1/q'}}\\
	&\le \nor{f}{\oo} \pars{ \frac{3}{2} \Delta + \frac{C\Delta^q }{\delta^q} } + \delta^{q'}  \mbb E \int_{s}^{t} |\alpha_r|^{q'}dr,
\end{align*}
and so
\[
	\sup_{\mu \in \mcl M} Z_\mu(s,t)
	\le m \pars{ \frac{3}{2} \Delta + C\Delta^q}.
\]
Raising both sides to the power $p$, taking the expectation $\mbf E$ over $\Omega$, and invoking the scaling properties of Brownian motion yield, for some constant  $C = C(K,p,q) > 0$ that changes from line to line,
\begin{align*}
	\mbf E \left[ \sup_{\mu \in \mcl M}  Z_\mu(s,t)^p \right]
	&\le C m^p \pars{ \mbf E \Delta^p + \mbf E \Delta^{pq} }\\
	&\le C m^p (t-s)^{p/2},
\end{align*}
and \eqref{pointwisemoment} then follows in this case.

{\it Case 2.} Assume now that
\begin{equation}\label{Deltacase2}
	t - s > \frac{\nor{f}{\oo}^q}{\nor{Df}{\oo}^q} \wedge \frac{\nor{f}{\oo}}{ \nor{\sigma \sigma^t}{\oo,[-1,0]} \nor{D^2 f}{\oo}}. 
\end{equation}
Set
\begin{equation}\label{h}
	h := \left[ \frac{ \nor{f}{\oo}}{\nor{Df}{\oo}} (t-s)^{1/q'}  \right] \wedge \frac{\nor{f}{\oo}}{ \nor{\sigma \sigma^t}{\oo} \nor{D^2 f}{\oo} }
\end{equation}
and let $N \in \NN$ be such that
\[
	\frac{t-s}{h} \le N < \frac{t-s}{h} + 1.
\]
Note that \eqref{Deltacase2} implies that $h \le t-s$, and so
\begin{equation}\label{Nh}
	t-s \le Nh < 2 (t-s)
\end{equation}

For $k = 0,1,2,\ldots, N-1$, set $\tau_k := s + kh$ and $\tau_N = t$, and, for $k = 1,2,\ldots, N$, define
\[
	\Delta_k = \max_{u,v \in [\tau_{k-1},\tau_k]} \abs{ B_u - B_v}.
\]
Using Lemma \ref{L:intzeta2}, we write
\begin{align*}
	\mbb E\left[ \int_{s}^{t} f(X_r)\cdot dB_r\right]
	&= \sum_{k=1}^{N} \mbb E \left[ \int_{\tau_{k-1}}^{\tau_k} f(X_r) \cdot dB_r \right] \\
	&= \mathrm{I} - \mathrm{II} - \mathrm{III},
\end{align*}
where
\[
	\mathrm{I} := \sum_{k=1}^N \mbb E\left[  f(X_{\tau_k}) \cdot (B_{\tau_k} - B_{\tau_{k-1}} )\right],
\]
\begin{align*}
	\mathrm{II} 
	&:= \sum_{k=1}^N \mbb E \left[\int_{\tau_{k-1}}^{\tau_k} Df(X_r)\alpha_r \cdot (B_r - B_{\tau_{k-1}})dr \right],
\end{align*}
and
\begin{align*}
	\mathrm{III}
	&:= \frac{1}{2} \sum_{k=1}^N \mbb E \left[\int_{\tau_{k-1}}^{\tau_k} \tr(\sigma_r \sigma_r^t D^2 f(X_r)) \cdot (B_r - B_{\tau_{k-1}})dr \right].
\end{align*}

We estimate
\[
	|\mathrm{I} | \le \nor{f}{\oo} \sum_{k=1}^N \Delta_k \quad \text{and} \quad |\mathrm{III}| \le \frac{h}{2} \nor{\sigma \sigma^t}{\oo} \nor{D^2 f}{\oo} \sum_{k=1}^{N} \Delta_k,
\]
and, for all $\eps > 0$, Young's inequality yields
\begin{align*}
	|\mathrm{II}| &\le \nor{Df}{\oo} \sum_{k=1}^{N} \Delta_k \mbb E  \int_{\tau_{k-1}}^{\tau_k} |\alpha_r|dr\\
	&\le \nor{Df}{\oo} h^{1/q} \sum_{k=1}^{N} \Delta_k \pars{  \mbb E  \int_{\tau_{k-1}}^{\tau_k} | \alpha_r|^{q'} dr}^{1/q'} \\
	&\le \nor{Df}{\oo} h^{1/q} \pars{ \frac{1}{q\eps^q} \sum_{k=1}^{N} \Delta_k^q +  \frac{\eps^{q'}}{q'}   \mbb E  \int_{s}^{t} |\alpha_r|^{q'}dr }.
\end{align*}
Combining the three estimates gives
\begin{equation}\label{unifst}
	\begin{split}
	\abs{\mbb E \left[ \int_{s}^{t} f(X_r) \cdot dB_r \right]}
	&\le \pars{\nor{f}{\oo} + \frac{h}{2} \nor{\sigma \sigma^t}{\oo} \nor{D^2 f}{\oo} } \sum_{k=1}^N \Delta_k \\
	&+ \nor{Df}{\oo} h^{1/q} \pars{ \frac{1}{q\eps^q} \sum_{k=1}^N \Delta_k^q +  \frac{\eps^{q'}}{q'}  \mbb E  \int_{s}^{t} | \alpha_r|^{q'}dr }.
	\end{split}
\end{equation}

We now set
\[
	\eps := \delta \pars{ \frac{q'}{\nor{Df}{\oo} h^{1/q}}}^{1/q'}.
\]
In particular,
\[
	\eps^{q'} = \frac{q'\delta^{q'}}{\nor{Df}{\oo} h^{1/q}} \quad \text{and} \quad \eps^q = \frac{(q')^{q-1}\delta^q}{\nor{Df}{\oo}^{q-1} h^{1/q'}},
\]
so that \eqref{unifst} becomes, for some $C = C(q) > 0$,
\begin{align*}
	\abs{ \mbb E \left[ \int_{s}^{t} f(X_r) \cdot dB_r  \right] } 
	&\le \pars{\nor{f}{\oo} + \frac{h}{2} \nor{\sigma \sigma^t}{\oo} \nor{D^2 f}{\oo} } \sum_{k=1}^N \Delta_k \\
	&+ \frac{C}{\delta^q} \nor{Df}{\oo}^q h \sum_{k=1}^N \Delta_k^q + \delta^{q'}   \mbb E  \int_{s}^{t} | \alpha_r|^{q'}dr .
\end{align*}
For $k = 1,2,\ldots,N$, the constants
\[
	a_k := \mbf E \Delta_k \quad \text{and} \quad b_k := \mbf E \Delta_k^q
\]
satisfy, for some $a > 0$ and $b = b(q) > 0$,
\[
	a_k \le ah^{1/2} \quad \text{and} \quad b_k \le b h^{q/2}.
\]
Then \eqref{h} and \eqref{Nh} give
\begin{align*}
	\pars{\nor{f}{\oo} + \frac{h}{2} \nor{\sigma \sigma^t}{\oo} \nor{D^2 f}{\oo} } \sum_{k=1}^N a_k
	&\le   \frac{3}{2} \nor{f}{\oo} Nh^{1/2} a \leq 3a \nor{f}{\oo}(t-s)h^{-1/2}\\
	& \le 3a(t-s)\nor{f}{\oo}^{1/2}  \left(\left[\nor{Df}{\oo}^{1/2}(t-s)^{-1/(2q')}\right]\vee \left[   \nor{\sigma \sigma^t}{\oo}^{1/2} \nor{D^2 f}{\oo}^{1/2}\right]\right)\\
	& \leq 3a K^{1/2} (t-s)^{1/2}
\end{align*}
and
\begin{align*}
	\nor{Df}{\oo}^q h \sum_{k=1}^N b_k &\le b \nor{Df}{\oo}^q N h^{1 + q/2}
	\le 2b(t-s) \nor{Df}{\oo}^q h^{q/2} \\
	&\le 2b(t-s) \nor{f}{\oo}^{q/2} \nor{Df}{\oo}^{q/2} (t-s)^{\frac{q}{2q'}} \le 2b K^{\frac{q}{2}} (t-s)^{\frac{q+1}{2}}.
\end{align*}
Therefore, because $0 < \delta \le 1$, we find that, for some constant $M_1 = M_1(K,q,m) > 0$,
\begin{equation}\label{almostpointwisemoment}
	\begin{split}
%	\sup_{\mu \in \mcl M} 
	&\pars{ Z_\mu(s,t)- M_1 (t-s)^{1/2}  }_+ \\
	&\qquad \le M_1 \pars{ \nor{f}{\oo} \abs{ \sum_{k=1}^N \pars{ \Delta_k - a_k} } +  C\nor{Df}{\oo}^q h\abs{ \sum_{k=1}^N (\Delta_k^q- b_k)} }.
	\end{split}
\end{equation}
The collections $(\Delta_k - a_k)_{k=1}^N$ and $(\Delta_k^q  -b_k)_{k=1}^N$ consist of independent and centered random variables. The scaling properties of Brownian motion yield, for any $k = 1,2,\ldots, N$ and $p_0 > 0$ and constants $A_1 = A_1(p_0) > 0$ and $A_2 = A_2(p_0,q) > 0$,
\[
	\mbf E \abs{ \Delta_k - a_k}^{p_0} \le A_1 h^{p_0/2} \quad \text{and} \quad \mbf E \abs{ \Delta_k^q - b_k}^{p_0} \le A_2 h^{p_0 q/2}.
\]
It is then a consequence of \eqref{Nh} and Lemma \ref{L:sums} that, for some constants $\tilde A_1 = \tilde A_1(p) > 0$ and $\tilde A_2 = \tilde A_2(p,q) > 0$,
\[
	\mbf E \abs{ \sum_{k=1}^N \pars{ \Delta_k - a_k} }^p \le \tilde A_1 N^{p/2} h^{p/2} \le 2^{p/2}\tilde A_1 (t-s)^{p/2}
\]
and
\[
	\mbf E \abs{ \sum_{k=1}^N (\Delta_k^q- b_k)}^p \le \tilde A_2 N^{p/2} h^{pq/2} \le 2^{p/2} \tilde A_2 (t-s)^{p/2}h^{p(q-1)/2}.
\]
The latter estimate and \eqref{h} give
\begin{align*}
	\nor{Df}{\oo}^{pq} h^{p} \mbf E \abs{ \sum_{k=1}^N (\Delta_k^q- b_k)}^p &\le 2^{p/2} \tilde A_2 \nor{Df}{\oo}^{pq} (t-s)^{p/2}h^{p(q+1)/2}\\
	&\le 2^{p/2} \tilde A_2 \nor{f}{\oo}^{p(q+1)/2} \nor{Df}{\oo}^{p(q-1)/2} (t-s)^{p \pars{ \frac{1}{2} + \frac{q+1}{2q'}} }\\
	&\le 2^{p/2} \tilde A_2 K^{p(q-1)/2} \nor{f}{\oo}^p (t-s)^{p/2},
\end{align*}
and so, raising \eqref{almostpointwisemoment} to the power $p$ and taking the expectation gives, for some $M_2 = M_2(m,K,p,q) > 0$,
\begin{align*}
	\mbf E &\left[ \sup_{\mu \in \mcl M} \pars{ Z_\mu(s,t) - M_1 (t-s)^{1/2}  }_+^p \right] \le M_2  (t-s)^{p/2}.
\end{align*}
Dividing by $(t-s)^{p/2}$ leads to \eqref{pointwisemoment}.

We now take $p$ large enough that
\[
	\kappa < \frac{1}{2} - \frac{1}{p}.
\]
Then \eqref{pointwisemoment} and Lemma \ref{L:Kolmogorov} imply that, for some $C = C(\kappa,m,K,p,q) > 0$ and $M = M(\kappa,m,K,q) > 0$, and for all $\lambda \ge 1$ and
\[
	p > \frac{2}{1 - 2\kappa},
\]
we have
\[
	\mbf P \pars{ \sup_{\mu \in \mcl M} \sup_{-1 \le s \le t \le 0} \frac{Z_\mu(s,t)}{(t-s)^\kappa} > M + \lambda} \le \frac{Cm^p}{\lambda^p}.
\]
By changing $C$ in a way that depends only on $m$ and $p$, the same can be accomplished for all $p \ge 1$. The proof is finished upon setting
\[
	\lambda_0 := 2M, \quad \mcl E := \sup_{\mu \in \mcl M} \sup_{-1 \le s \le t \le 0} \frac{Z_\mu(s,t)}{(t-s)^\kappa},
\]
and replacing $C$ with $2^p C$.
\end{proof}

\bibliography{regHJstochforced}{}
\bibliographystyle{acm}

\end{document}